\documentclass[11pt]{amsart}
\usepackage{amssymb, amsmath, latexsym, amsthm, amscd, fancyhdr, graphicx, xcolor, hyperref, bm}
\def\R{\mathbb R}
\def\N{\mathbb N}
\def\C{\mathbb C}
\def\Z{\mathbb Z}

\def\F{\mathcal F}

\def\H{\mathbb H}

\def\={\equiv}
\def\<{\langle}
\def\>{\rangle}
\def\eps{\varepsilon}

\def\ov{\overline}

\def\inv{^{-1}}

\def\supp{\operatorname{supp}}
\def\stab{\operatorname{Stab}}

\def\bms{m^{\operatorname{BMS}}}
\def\r{\textbf{r}}
\def\v{\textbf{v}}
\def\w{\textbf{w}}
\def\x{\textbf{x}}
\def\y{\textbf{y}}
\def\z{\textbf{z}}
\def\q{\textbf{q}}
\def\K{\mathcal{K}}
\def\1{\mathbf{1}}
\def\SO{\operatorname{SO}}
\def\Up{\Upsilon}
\def\up{\upsilon}
\def\btau{{\bm\tau}}
\def\bsig{\bm{\sigma}}
\def\tup{\tilde{\Upsilon}}
\def\bp{\operatorname{bp}}
	
\def\d{\textbf{d}}
\def\f{\textbf{f}}
\def\i{\textbf{i}}
\def\b{\textbf{b}}
\def\t{\textbf{t}}
\def\s{\textbf{s}}

\def\ps{\mu^{\operatorname{PS}}}
\def\br{m^{\operatorname{BR}}}
\newtheorem{theorem}{Theorem}[section]

\newtheorem{proposition}[theorem]{Proposition}
\newtheorem{corollary}[theorem]{Corollary}
\newtheorem{lemma}[theorem]{Lemma}
\newtheorem{definition}[theorem]{Definition}

\newtheorem{remark}[theorem]{Remark}

\makeatletter
\let\oldtocsection=\tocsection
\let\oldtocsubsection=\tocsubsection

\renewcommand{\tocsection}[2]{\hspace{0em}\oldtocsection{#1}{#2}}
\renewcommand{\tocsubsection}[2]{\hspace{1em}\oldtocsubsection{#1}{#2}}
\def\subsubsection{\@startsection{subsubsection}{3}%
	\z@{.5\linespacing\@plus.7\linespacing}{-.5em}%
	{\normalfont\bfseries}}
\makeatother

\numberwithin{equation}{section}

\title[Factors and joinings]{On factor rigidity and joining classification for infinite volume rank one homogeneous spaces}

\author[J. M. Warren]{Jacqueline M. Warren}
\address[J. M. Warren]{Department of Mathematics, University of California, San Diego}
\thanks{J. M. Warren was supported in part by an NSERC PGS-D3 fellowship.}

\begin{document}
	\maketitle
	
	\begin{abstract}
		We classify locally finite joinings with respect to the Burger-Roblin measure for the action of a horospherical subgroup $U$ on $\Gamma \backslash G$, where $G = \operatorname{SO}(n,1)^\circ$ and $\Gamma$ is a convex cocompact and Zariski dense subgroup of $G$, or geometrically finite with restrictions on critical exponent and rank of cusps. %This extends the classification by Oh and Mohammadi obtained in the case that $G = \operatorname{PSL}_2(\R)$ or $\operatorname{PSL}_2(\C)$ and $\Gamma$ is geometrically finite and Zariski dense.
		
		We also prove in the more general case of $\Gamma$ geometrically finite and Zariski dense that certain $U$-equivariant set-valued maps are rigid.%, generalizing results of Ratner in the finite volume case and of Flaminio and Spatzier for factor maps in the convex cocompact case.
	\end{abstract}

\tableofcontents

\section{Introduction}

In \cite{ratner}, Ratner classified all joinings with respect to horocycle flows on finite volume quotients of $\operatorname{PSL}_2(\R)$, a problem which is closely related to that of measure classification and classification of closed orbits. These problems are well understood in the finite volume case, \cite{ratneracta, ratnerannals, witte}, but for infinite volume, a full picture is not yet clear. 

In \cite{OMM, OMM2, OMM3}, McMullen, Mohammadi, and Oh obtained orbit closure classification in the infinite volume setting, specifically for convex cocompact acyclindral 3-manifolds. This was generalized to higher dimensions by Lee and Oh in \cite{LeeOh}, and to geometrically finite acylindical manifolds in \cite{OhBenoist} by Benoist and Oh. 

Mohammadi and Oh also show equidistribution for non-closed orbits in the geometrically finite case in \cite{joinings}, as well as classifying joinings for geometrically finite quotients of $\operatorname{PSL}_2(\R)$ or $\operatorname{PSL}_2(\C)$.

There is some progress in the geometrically infinite case under certain assumptions as well. For instance, in \cite{ledrappier2}, Ledrappier and Sarig classify measures for geometrically infinite regular covers of compact hyperbolic surfaces, and in \cite{pan}, Pan classifies joinings for $\Z$ or $\Z^2$ covers of compact hyperbolic surfaces.

In addition, the following works are in the setting of either geometrically finite quotients or certain geometrically infinite quotients: \cite{bowditch, burger, roblin}. In particular \cite{ledrappier1,  sarig1, winter} consider the problem of measure classification, while \cite{EL} considers joinings in higher rank. 

The primary purpose of this paper is to extend the classification of joinings from \cite{joinings}, and the proof, a la Ratner, will require proving rigidity of $U$-equivariant set-valued maps, a result which is of independent interest. Ratner proved this for factor maps in the lattice case in \cite{ratnerfactors}, while Flaminio and Spatzier proved a similar result for convex cocompact groups in \cite{FSinventiones}. For set-valued maps, rigidity was proven in the finite volume case by Ratner in \cite{ratner}, where she calls these maps measurable partitions. A partial statement is proven in the case of geometrically finite quotients of $\SO(n,1)^\circ$ for $n=2,3$ in \cite{joinings}. We will prove a full rigidity result for such set-valued maps for arbitrary $n$. We prove rigidity for factor maps with no further assumptions; for set-valued maps, we need to assume the presence of a ``nice'' ergodic measure. 

Recall that $G:=\SO(n,1)^\circ$ corresponds to the group of orientation preserving isometries of real hyperbolic space $\H^n$. Throughout the paper, we assume that \begin{center}$\Gamma_1, \Gamma_2$ are geometrically finite and Zariski dense discrete subgroups of $G$ with infinite co-volume\end{center} and define \[X_i := \Gamma_i \backslash G \text{ for }i=1,2 \text{ and } X := X_1 \times X_2.\] Our factor rigidity statement will hold for geometrically finite groups, but our proof of the joining classification will require assuming either that the $\Gamma_i$'s are convex cocompact, or are geometrically finite with all cusps of full rank and critical exponents larger than $n-\frac{5}{4}$.

Let $U$ denote a horospherical subgroup of $G$. In the infinite volume setting, the natural analogue of the Haar measure used in Ratner's proofs in the finite volume case is the \emph{Burger-Roblin} (BR) \emph{measure} $\br_i = \br_{\Gamma_i}$ (defined in \S\ref{section; bms and br}). The BR measure is the unique locally finite $U$-ergodic measure that is not supported on a closed $U$-orbit, \cite{burger, roblin, winter}.

Our rigidity theorem for factor maps is stated below. A more general version that covers set-valued maps, Theorem \ref{thm; full rigidity thm}, is proven in \S\ref{section; rigidity}, but requires further assumptions. Here, the action of $A$ induces the frame flow on $\Gamma_i \backslash \H^n$, $M$ is the compact centralizer of $A$, and $U^-$ is the opposite horospherical subgroup. %$\bms_i:=\bms_{\Gamma_i}$ denotes the \emph{Bowen-Margulis-Sullivan} (BMS) \emph{measure}, which is defined in \S\ref{section; bms and br}. We will often omit the subscript.

\begin{theorem}
Let $$\Up:X_2\to X_1$$ be a measurable map, and suppose that there exists a $\br$-conull set on which $\Up$ is $U$-equviariant. Then there exists a map $$\hat{\Up}:X_2 \to X_1,$$ a constant $\bsig_0 \in U$, and a $U$-invariant $\br$-conull set $X_2'$ such that for all $x \in X_2'$, \begin{itemize}
	\item $\hat{\Up}(x) = \Up(x)\bsig_0$,
	\item $\hat{\Up}(xh) = \hat{\Up}(x)h$ for all $h \in AMU$, and
	\item $\hat{\Up}(xv) = \hat{\Up}(x)v$ for all $v \in U^-$ such that $xv \in X_2'$.
\end{itemize}
\label{thm; rigidity only for factors}
\end{theorem}

%Note that $\{\text{cardinality } \ell \text{ subsets of } X_1\}$ is a complete metric space with the Hausdorff metric.

%This theorem is broken down into two more precise statements, Theorems \ref{thm; rigidity; AM part main thm} and \ref{thm;rigidity;U- main thm}, which are proved in \S\ref{section; rigidity}. 
Our proofs follow the outline of Ratner's approach from \cite{ratner}, but care is required in this infinite volume case.

%For the joining classification in \S\ref{section; joining classification}, we will further assume that the $\Gamma_i$'s are convex cocompact or geometrically finite with all cusps of rank $n-1$ and critical exponents greater than $n-\frac{5}{4}$.

Denote by $\Delta(U)$ the diagonal embedding into $G \times G$.

\begin{definition}\label{defn; joining}
	Let $\mu_i$ be a locally finite $U$-invariant Borel measure on $\Gamma_i\backslash G$ for $i=1,2$. A U-\emph{joining} with respect to $(\mu_1,\mu_2)$ is a locally finite $\Delta(U)$-invariant measure $\mu$ on $X$ such that the push-forward measure onto the $i$-th coordinate is proportional to the corresponding $\mu_i$, $i=1,2$. If $\mu$ is $\Delta(U)$-ergodic, we call it an \emph{ergodic $U$-joining}.	
\end{definition}

The primary goal of this article is to classify $U$-joinings in this infinite volume setting for the pair $(\br_1, \br_2)$. Note that in this case, the $\br_i$'s are infinite measures \cite{br infinite}, so the product measure $\br_1 \times \br_2$ is not a $U$-joining. 

We now restate the definition of a \emph{finite cover self-joining} as it appears in \cite{joinings}:

\begin{definition}\label{defn; finite cover self joining}
	Suppose that there exists $g_0 \in G$ so that $g_0\inv \Gamma_1 g_0$ and $\Gamma_2$ are commensurable in $G$. In particular, we have an isomorphism $$(g_0\inv \Gamma_1 g_0\cap\Gamma_2)\backslash G \to [(g_0, 1_G)]\Delta(G)$$ defined by $$[g] \mapsto [(g_0g, g)],$$ where $1_G$ denotes the identity in $G$. The pushforward of the $\operatorname{BR}$ measure $\br_{g_0\inv \Gamma_1 g_0\cap\Gamma_2}$ is a $U$-joining, which we call a \emph{finite cover self-joining}. We also consider any translation of a finite cover self-joining under an element of the form $(u,1_G) \in U \times \{1_G\}$ to be a finite cover self-joining.
\end{definition}

We will obtain the following joining classification:

\begin{theorem}\label{thm; final classification}
	Suppose that $\Gamma_1, \Gamma_2$ are \begin{itemize}
		\item convex cocompact, or
		\item geometrically finite with all cusps full rank and critical exponents greater than $n- \frac{5}{4}$.
	\end{itemize} Then every locally finite ergodic $U$-joining on $X = \Gamma_1 \backslash G \times \Gamma_2 \backslash G$ with respect to $(\br_1, \br_2)$ is a finite cover self-joining.
	
	In particular, $X$ admits a $U$-joining if and only if, up to a conjugation, $\Gamma_1$ and $\Gamma_2$ are commensurable.
\end{theorem}

Note that by the ergodic decomposition and the $U$-ergodicity of $\br$ \cite{winter}, classifying the ergodic $U$-joinings is sufficient to understand all $U$-joinings.

The proof strategy involves first reducing to a specific $U$-equivariant set-valued map, for which we prove an analogue of Theorem \ref{thm; rigidity only for factors} in \S\ref{section; rigidity}. In \cite{joinings}, Mohammadi and Oh prove rigidity of set-valued maps that are equivariant under a subgroup of $AM$ and under $U$ in the geometrically finite setting. Because of this extra equivariance assumption, their joining classification argument in \cite[Section 7]{joinings} requires arguing that a $U$-joining will be invariant under a subgroup of $AM$ before reducing the problem to rigidity of a certain set-valued map. With the more general rigidity result in \S\ref{section; rigidity}, when proving joining classification, we avoid the need for such an argument.

This article is organized as follows. In \S\ref{section; prelims}, we define notation that is used throughout the paper, the Patterson-Sullivan (PS), Bowen-Margulis-Sullivan (BMS) and Burger-Roblin (BR) measures, and reference some basic properties of these measures. 

In \S\ref{section; rigidity}, we prove Theorem \ref{thm; rigidity only for factors}, by proving the more general Theorem \ref{thm; full rigidity thm}, which includes set-valued maps. In particular, we prove that, under certain assumptions, a $U$-equivariant set-valued map must also be $AMU^-$-equivariant, up to a constant shift by an element of $U$, Theorems \ref{thm; rigidity; AM part main thm} and \ref{thm;rigidity;U- main thm}. This will be key in the final steps of the proof of Theorem \ref{thm; final classification}. Note that the results in this section are proved in the more general setting of $\Gamma_i$'s being geometrically finite and Zariski dense, not necessarily convex cocompact, although in general, for set-valued maps, we need to assume the existence of an ergodic joining-like measure on $X$, see Theorem \ref{thm; full rigidity thm}.

In \S\ref{section; varieties}, we prove general results about the behaviour of PS measure on varieties that will be important in the proof of Theorem \ref{thm; final classification}. In particular, we prove that Lebesgue integrals on small neighbourhoods of varieties are controlled by the PS measure, Lemma \ref{lem;int over V small, abs cty}. Understanding this behaviour is a crucial step needed to generalize the results from \cite{joinings} to higher dimensions.

In \S\ref{section; joining classification}, we show that the fibers of the projection $\pi_2$ onto the second coordinate must be finite, Theorem \ref{thm; 717}. This allows us to use the results of \S\ref{section; rigidity} to prove Theorem \ref{thm; finite cover}, which is a more precise formulation of Theorem \ref{thm; final classification}.

\section{Preliminaries and notation}\label{section; prelims}

For convenience, we remind the reader of the following notation that appeared in the introduction: \begin{itemize}
	\item $G = \operatorname{SO}(n,1)^\circ$ is the connected component of the identity in $\operatorname{SO}(n,1)$. It is the group of orientation preserving isometries of $\H^n$, and $1_G$ denotes the identity element in $G$.
	\item $\Gamma_1, \Gamma_2$ are geometrically finite and Zariski dense discrete subgroups of $G$ with infinite co-volume.
	\item $X_i := \Gamma_i \backslash G$ and $X:= X_1 \times X_2$.
	\item For $H \subset G$, $\Delta(H)$ denotes the diagonal embedding of $H$ into $G \times G$.
\end{itemize}

Define \[A = \{a_s : s \in \R\} \text{ where } a_s = \begin{pmatrix} e^{s}  & & \\ & I_{n-1} & \\ & & e^{-s}\end{pmatrix},\] where $I_{n-1}$ is the $(n-1)\times(n-1)$ identity matrix. 

We denote by $U$ the contracting horospherical subgroup, that is, \[U= \{g \in G : a_s g a_{-s} \to 1_G \text{ as } s \to \infty\}.\] Similarly, we denote by $U^-$ the expanding horospherical subgroup, \[U^- = \{g \in G : a_{s}ga_{-s} \to 1_G \text{ as } s \to -\infty\}.\] Both groups are isomorphic to $\R^{n-1}$, and we use the following parametrizations:

\[U = \{u_\t : \t \in \R^{n-1}\} \text{ where }u_\t = \begin{pmatrix} 1 & & \\ \t^T & I & \\ \frac{1}{2}|\t|^2 & \t & 1 \end{pmatrix} \text{ and }\]
\[U^- = \{v_\t : \t \in \R^{n-1}\} \text{ where } v_\t = \begin{pmatrix} 1 & \t & \frac{1}{2}|\t|^2 \\ & I & \t^T \\ & & 1\end{pmatrix}.\]

We also define \[M = \left\{\begin{pmatrix} 1 & & \\ & m & \\ & & 1 \end{pmatrix} : m \in \operatorname{SO}(n-1)\right\}.\] We will often abuse notation by writing $m \in M$ to refer to the matrix $\begin{pmatrix} 1 & & \\ & m & \\ & & 1 \end{pmatrix}$. Observe that these parametrizations satisfy \[a_{-s}u_\t a_s = u_{\t e^s} \text{ and } m\inv u_\t m = u_{\t m}.\]

We view $G$ as being embedded in $\operatorname{SL}_{n+1}(\R)$. For $T > 0$, we denote balls in $G$ by \[B(T) = \{g \in G : \|g-I\| \le T\} \text{ where } \|\cdot\| \text{ is the max norm on }G,\] and in $U$ by \[B_U(T) = \{u_\t \in U : |\t|\le T\}, \text{ where } |\cdot | \text{ is the max norm on }\R^{n-1}.\] We write $\t \in B_U(T)$ as shorthand for $u_\t \in B_U(T)$. 

On $\Gamma \backslash G$, \begin{equation}
d(\Gamma x, \Gamma y) = \min\{\|g-1_G\| : g \in G, \Gamma x = \Gamma y g\}.\label{def; def of d}
\end{equation}

\subsection{PS measure}\label{section; ps}

We use many definitions and notations as in \cite[Section 2]{joinings}, but provide a paraphrased version here for the convenience of the reader. See \cite[Section 2]{joinings} for more details about these constructions.

Let $\partial(\H^n) \cong \mathbb{S}^{n-1}$ denote the geometric boundary of $\H^n$. For any discrete subgroup $\Gamma$ of $G$, we can define the \emph{limit set} of $\Gamma$, $\Lambda(\Gamma)$, as the accumulation points of any orbit in $\partial(\H^n)$, that is, $$\Lambda(\Gamma) = \ov{\Gamma v}-\Gamma v$$ for any $v \in \H^n$, where the closure is taken in $\H^n \cup \partial(\H^n)$. This is independent of $v$ because $\Gamma$ acts by isometries on $\H^n$. 

We denote by $\Lambda_{\operatorname{r}}(\Gamma)$ the set of \emph{radial limit points} of $\Gamma$. $\xi \in \Lambda(\Gamma)$ is a radial limit point if some (hence every) geodesic ray towards $\xi$ has accumulation points in some compact subset of $\Gamma\backslash G$. A \emph{parabolic} limit point $\xi \in \Lambda(\Gamma)$ is one that is fixed by a parabolic element of $\Gamma$, that is, some element of $\Gamma$ that fixes exactly one element of $\partial(\H^n)$. A parabolic limit point $\xi \in \Lambda(\Gamma)$ is called \emph{bounded parabolic} if the stabilizer of $\xi$ in $\Gamma$ acts cocompactly on $\Lambda(\Gamma) - \{\xi\}$. We denote the set of bounded parabolic limit points by $\Lambda_{\bp}(\Gamma)$. In the case of $\Gamma$ convex cocompact, $$\Lambda(\Gamma) = \Lambda_{\operatorname{r}}(\Gamma).$$ If $\Gamma$ is geometrically finite, then  $$\Lambda(\Gamma) = \Lambda_{\operatorname{r}}(\Gamma) \cup \Lambda_{\bp}(\Gamma),$$ and $\Lambda_{\bp}(\Gamma)$ is countable, \cite{bowditch}. 

Fix a reference point $o \in \H^n$ and a reference vector $w_o \in T_o^1(\H^n)$, the unit tangent space of $\H^n$ at $o$. Consider the maximal compact subgroup $K := \stab_G(o)$. Then $\H^n$ can be viewed as $G/K$. Similarly, $M$ is $\stab_G(w_0)$, and $T^1(\H^n)$ can be identified with $G/M$. 

With these identifications and the parametrizations in \S\ref{section; prelims}, $A$ implements the geodesic flow on $T^1(\H^n)$. That is, if $\{g^t : t \in \R\}$ is the geodesic flow on $T^1(\H^n)$, then $g^t(w_0) = [a_s M]$, where $[\cdot]$ denotes the coset in $G/M$. 

For $w \in T^1(\H^n)$, $w^\pm \in \partial(\H^n)$ denotes the forward or backward endpoints of the geodesic $w$ determines, i.e. $w^\pm = \lim\limits_{t \to \infty} g^t(w)$. For $g \in G$, we define \[g^\pm := gw_0^\pm.\] For $x = \Gamma g \in \Gamma \backslash G$, we write $x^\pm\in \Lambda(\Gamma)$ if $g^\pm \in \Lambda(\Gamma)$ for some representative of the coset. This is well-defined by definition of $\Lambda(\Gamma)$.

Let $x, y \in \H^n$ and $\xi \in \partial(\H^n)$. The \emph{Busemann function} based at $\xi$ is the function $$\beta_\xi(x,y) = \lim\limits_{t \to \infty} d(\xi_t, x) - d(\xi_t,y),$$ where $d$ is the hyperbolic metric and $\xi_t$ is a geodesic ray in $\H^n$ towards $\xi$.

For $\Gamma < G$ discrete, a $\Gamma$\emph{-invariant conformal density of dimension} $\delta>0$ is a family $\{\mu_x : x \in \H^n\}$ of pairwise mutually absolutely continuous finite measures on $\partial(\H^n)$ satisfying \[\gamma_*\mu_x = \mu_{\gamma x} \text{ and } \frac{d\mu_x}{d\mu_y}(\xi) = e^{-\delta \beta_{\xi}(x,y)}\] for all $x,y \in \H^n$ and $\xi \in \partial(\H^n)$, where $\gamma_*\nu_x(E) = \nu_x(\gamma\inv(E))$ for all Borel measurable subsets $E \subseteq \partial(\H^n)$.

Let $\delta_\Gamma$ denote the critical exponent of $\Gamma$. Up to scalar multiplication, there exists a unique $\Gamma$-invariant conformal density of dimension $\delta_\Gamma$, denoted by $\{\nu_x : x \in \H^n\}$, and called the \emph{Patterson-Sullivan density}.

For each $g \in G$, this density allows us to define the \emph{Patterson-Sullivan} (PS) \emph{measure} on a horocycle $gU$ by \[d\ps_{gU}(gu_\t) = e^{\delta_\Gamma \beta_{(gu_\t)^+}(o, gu_\t(o))} d\nu_o(gu_\t)^+.\] If $x^- \in \Lambda_{\operatorname{r}}(\Gamma)$, then the map $u \mapsto xu$ is injective on $U$, and we can define the PS measure on $xU \subseteq \Gamma \backslash G$ by push forward. However, in general, there is some subtlety in defining the PS measure on $xU$; see \cite[Section 2.3]{joinings} for more discussion of this. We note that $\ps_{gU}$ can be viewed as a measure on $U \cong \R^{n-1}$ via $d\ps_g(\t) = d\ps_{gU}(gu_\t)$.

The \emph{Lebesgue density} is $\{m_x : x \in \H^n\}$, where $m_x$ is the unique probability measure on $\partial(\H^n)$ that is invariant under $\stab_G(x)$. The Lebesgue density is a $G$-invariant conformal density of dimension $n-1$. We similarly define the Lebesgue measure on $gU$: \[d\mu_{gU}^{\operatorname{Leb}}(gu_\t)=e^{(n-1)\beta_{(gu_\t)^+}(o, gu_\t(o))} dm_o(gu_\t)^+.\] This is independent of the orbit and is in fact a scalar multiple of the Lebesgue measure on $U \cong \R^{n-1}$, denoted by $d\t$.

Note that for every Borel measurable subset $E \subseteq U$, every $g\in G$, and every $s \in \R$, the properties of conformal densities imply that \[\ps_g(E) = e^{\delta_\Gamma s}\ps_{ga_{-s}}(a_s E a_{-s}).\] In particular, \[\ps_g(B_U(e^s)) = e^{\delta_\Gamma s}\ps_{ga_{-s}}(B_U(1)).\]

We record the following properties of PS measure:

\begin{lemma}\cite[Cor. 1.4]{FSinventiones}
	 For every $g \in G$, every proper subvariety of $U$ is a null set for $\ps_g$.
\label{lem; vars null sets, cty of g to psg}
\end{lemma} 

%The following lemma follows from the definition of the PS measure, because it is defined using stereographic projection and the Busemann function.

\begin{lemma}%\cite[Lemma 4.1, Cor. 4.2]{FSinventiones}
	The map $g \mapsto \ps_{gU}$ is continuous, where the topology on the space of all regular Borel measures on $U$ is given by $\mu_n \to \mu \iff \mu_n(f) \to \mu(f)$ for all $f \in C_c(U)$.
	\label{lem;new cty ps measure}
\end{lemma} 
\begin{proof}
	The proof follows by the definition of the PS measure, since it is defined using stereographic projection and the Busemann function, which are continuous.
\end{proof}

\begin{corollary}\cite[Cor. 2.2]{joinings}
	For any compact set $\Omega \subset G$ and any $T>0$, \[0 < \inf\limits_{g \in \Omega, g^+ \in \Lambda(\Gamma)} \ps_g(B_U(T)) \le \sup\limits_{g \in \Omega, g^+ \in \Lambda(\Gamma)} \ps_g(B_U(T)) < \infty.\]
	\label{cor; inf sup of balls over compact}
\end{corollary}

\begin{lemma}
	For every compact subset $\Omega \subseteq X_2$, there exists $\kappa = \kappa(\Omega)>0$ such that $$0<\inf_{x\in \Omega}\ps_x(B_U(\kappa))\le \sup_{x \in \Omega}\ps_x(B_U(\kappa))<\infty.$$
	\label{lem;PS inf kappa for compact set}
\end{lemma}
\begin{proof}
	Because $G \ni g\mapsto d_{\partial\H^n}(g^+,\Lambda(\Gamma))$ is continuous, there exists $\kappa > 0$ such that $$(gB_U(\kappa/2))^+ \cap \Lambda(\Gamma) \ne \emptyset$$ for every $x = \Gamma g \in \Omega$, and thus $\ps_x(B_U(\kappa))>0$ for all $x \in \Omega$. 
	
	Since $x \mapsto \ps_x$ is also continuous, it follows that $\{\ps_x(B_U(\kappa)):g\in \Omega\}$ is bounded with a positive lower bound.
\end{proof}

\subsection{BMS and BR measures}\label{section; bms and br}

%This section is again largely a condensed version of \cite[Section 2.4]{joinings}. 

The map $w \mapsto (w^+, w^-, \beta_{w^-}(o, \pi(w)))$, where $\pi(w) \in \H^n$ is the base point of $w$, is a homeomorphism between $T^1(\H^n)$ and $$(\partial(\H^n) \times \partial(\H^n) - \{(\xi,\xi) : \xi \in \partial(\H^n)\}) \times \R.$$ This identification allows us to define the BMS and BR measures on $T^1(\H^n) \cong G/M$, denoted by $\tilde{m}^{\operatorname{BMS}}$ and $\tilde{m}^{\operatorname{BR}}$:  \begin{equation}d\tilde{m}^{\operatorname{BMS}}(w) = e^{\delta_\Gamma (\beta_{w^+}(o, \pi(w)) + \beta_{w^-}(o,\pi(w)))}d\nu_o(w^+)d\nu_o(w^-)ds\label{eqn;defn of bms}\end{equation}
\begin{equation}d\tilde{m}^{\operatorname{BR}}(w) = e^{(n-1)\beta_{w^+}(o, \pi(w)) + \delta_\Gamma \beta_{w^-}(o, \pi(w))}dm_o(w^+)d\nu_o(w^-)ds.\label{eqn;defn of br}\end{equation}

By lifting to $M$-invariant measures on $G$, these will induce locally finite Borel measures on $\Gamma \backslash G$, denoted by $\bms$ and $\br$; see \cite[Section 2.4]{joinings} for more details. $\bms$ is a finite measure \cite{sullivan} (which we will assume to be normalized to a probability measure) and $\br$ is infinite if and only if $\Gamma$ is not a lattice, \cite{br infinite}.

We have that $$\supp\bms = \{x \in \Gamma \backslash G : x^\pm \in \Lambda(\Gamma)\}$$ and $$\supp\br = \{x \in \Gamma \backslash G : x^- \in \Lambda(\Gamma)\}.$$ Convex cocompactness is equivalent to $\supp\bms$ being compact.

The relationship between $\bms$ and $\br$ will be important in what follows. Note that by comparing equations (\ref{eqn;defn of bms}) and (\ref{eqn;defn of br}), we see that the most significant difference is in the appearance of $d\nu_o(w^+)$ vs. $dm_o(w^+)$, that is, the major difference is in how they see the $U$ direction. In particular, $\br$ is $U$-ergodic when $\Gamma$ is Zariski dense \cite{winter}, while $\bms$ is not even $U$-invariant. Moreover, \begin{equation}\text{if a set } E \text{ is }U\text{-invariant, then }\bms(E)=1 \iff \br(E^c)=0.\label{eqn;reln bt br and bms}\end{equation}

\section{Rigidity of $U$-equivariant set-valued maps}
\label{section; rigidity}

In this section, we assume only that the $\Gamma_i$'s are geometrically finite and Zariski dense subgroups of $G$. We will prove the following theorem, which includes Theorem \ref{thm; rigidity only for factors} in the first case. Note that the codomain of $\Up$ in the following statement is a complete metric space with the Hausdorff metric.

\begin{theorem} Let $\ell \in \N$, and suppose that $$\Up:X_2 \to \{\text{cardinality } \ell \text{ subsets of } X_1\}$$ is a measurable map and is $U$-equivariant on a $\br$-conull set. Suppose further that either \begin{enumerate}
		\item $\ell =1$, or
		\item there exists a $\Delta(U)$-ergodic measure $\mu$ on $X=X_1 \times X_2$ such that \begin{itemize}
			\item if $Z \subseteq X_2$ is $\br$-conull, then $\bigcup\limits_{x_2\in Z}(\Up(x_2)\times \{x_2\})$ is $\mu$-conull, and
			\item if $W \subseteq X$ is $\mu$-conull, then $\pi_2(W)$ is $\br$-conull.
		\end{itemize}
	\end{enumerate}
	Then there exists a map $$\hat{\Up}:X_2 \to \{\text{cardinality } \ell \text{ subsets of } X_1\},$$ a constant $\bsig_0 \in U$, and a $U$-invariant $\br$-conull set $X_2'$ such that for all $x \in X_2'$,\begin{itemize}
		\item $\hat{\Up}(x) = \Up(x)\bsig_0$,
		\item $\hat{\Up}(xh) = \hat{\Up}(x)h$ for all $h \in AMU$, and
		\item $\hat{\Up}(xv) = \hat{\Up}(x)v$ for all $v \in U^-$ such that $xv \in X_2'$.
	\end{itemize}
	\label{thm; full rigidity thm}
\end{theorem}

Throughout this section, suppose that $\ell$ and  $\Up$ are as in Theorem \ref{thm; full rigidity thm}, and that $\tilde{X}_2 \subseteq X_2$ is such that $$\Up(xu)=\Up(x)u$$ for all $u \in U$ and $x \in \tilde{X}_2$. 

%Let $\ell \in \N$ and suppose that $$\Up:X_2 \to \{\text{cardinality }\ell\text{ subsets of }X_1\}$$ is a measurable map (the right hand side is a complete metric space with the Hausdorff metric) such that there exists a $U$-invariant $\br$-conull set $\tilde{X_2}\subseteq X_2$ with $$\Up(xu)=\Up(x)u$$ for all $u \in U$ and $x \in \tilde{X}_2$.
By a standard argument for constructing measurable cross-sections, we may also assume that there exist measurable maps $\up_i:\tilde{X_2} \to X_1$ such that $$\Up(x) = \{\up_1(x),\ldots, \up_\ell(x)\}$$ for all $x \in \tilde{X_2}.$ The following construction shows that we can further assume that the maps $\up_i$ are defined $\bms$-a.e.:\ let $\{B_n:n\in\N\}$ be a countable collection of balls that cover $X_2$ and satisfy $\bms(\partial B_n)=0$. Proceed inductively: for $x \in( B_n\cap AMU^- \cap \tilde{X}_2 )- B_{n-1}$ and $u$ so that $xu \in B_n - B_{n-1}$, define $$\up_i(xu)=\up_i(x)u.$$ This shows that the $\up_i$'s can be measurably defined on $\tilde{X}_2U - \bigcup\limits_{n=1}^\infty \partial B_n$, a $\bms$-conull set.

We will assume throughout this section that either \begin{enumerate}
	\item $\ell =1$, or
	\item there exists a $\Delta(U)$-ergodic measure $\mu$ on $X=X_1 \times X_2$ such that \begin{itemize}
		\item if $Z \subseteq X_2$ is $\br$-conull, then $\bigcup\limits_{x_2\in Z}(\Up(x_2)\times \{x_2\})$ is $\mu$-conull, and
		\item if $W \subseteq X$ is $\mu$-conull, then $\pi_2(W)$ is $\br$-conull.
	\end{itemize}
\end{enumerate}

\begin{remark}
	\emph{These further assumptions are needed only in the proof of Lemma \ref{lem; defn Wh}.}
\end{remark}

%\begin{equation}
%\label{def;br conull implies mu conull} \text{if } Z \subseteq X_2 \text{ is } \br\text{-conull, then } \bigcup\limits_{x_2 \in Z}(\Up(x_2)\times\{x_2\}) \text{ is } \mu\text{-conull, and}
%\end{equation}
%\begin{equation}
%\label{def;mu conull implies pi2 br conull} \text{if } W \subseteq X \text{ is } \mu\text{-conull, then } \pi_2(W) \text{ is } \br\text{-conull},
%\end{equation} where $\pi_2: X \to X_2$ is the natural projection map.

\begin{remark}
	\emph{In our application to joining classification in \S\ref{section; joining classification}, the conditions in case (2) will be satisfied with $\mu$ an ergodic $U$-joining for $(\br_1,\br_2)$ and $$\Up(x_2):= \pi_1(\pi_2\inv(x_2)).$$ That this $\Up$ will take values in cardinality $\ell$ subsets of $X_1$ for some $\ell \in \N$ is not immediately clear, and this is proven in \S\ref{section; fibers finite}.}
\end{remark}

\S\ref{section; AMU equiv} is dedicated to the proof of the following theorem:

\begin{theorem} There exists a $U$-invariant $\br$-conull set $Y \subseteq X_2'$, a constant $\bsig_0 \in U$, and $\tilde{\Up}:X_2 \to \{\text{cardinality }\ell \text{ subsets of } X_1\}$ satisfying: \begin{enumerate}
		\item $\tilde{\Up} =\Up$ on $Y$,
		\item for all $x \in Y,$ and all $h \in AMU$, $\tilde\Up(xh)u_{\bsig_0} = \tilde{\Up}(x)u_{\bsig_0}h.$
	\end{enumerate} \label{thm; rigidity; AM part main thm} 
\end{theorem}

The heart of the argument lies in the following proposition.

\begin{proposition}
	For all sufficiently small $\eta>0$, there exists $\ov\eps>0$ with $\ov\eps\to 0$ as $\eta\to 0$ and a $U$-invariant $\br$-conull set $\hat{X}_h$ satisfying: for every $h \in B_{AM}(\eta)$, every $x \in \hat{X}_h$, and every $1 \le i \le \ell$, there exists a unique $1 \le k(i) \le \ell$ and $\btau_h(x,\up_i(x)) \in B_U(\ov\eps)$ such that $$\up_{k(i)}(xh) = \up_i(x)u_{\btau_h(x,\up_i(x))}h.$$ Moreover, $\btau_h(x,\up_i(x)) = \btau_h(xu_\t, \up_i(x)u_\t)$ for all $\t \in U$.%, and $h \mapsto \btau_h(x,\up_i(x))$ is uniformly continuous.
	\label{prop;existence of tau}
\end{proposition}

An analogous statement is proven for convex cocompact groups and assuming $\Up$ is a factor map in \cite{FSinventiones}, following Ratner's approach from \cite{ratnerfactors}. In \cite{ratner}, Ratner proves this in the case of finite volume set-valued maps, which she refers to as measurable partitions; we follow the general lines of her approach in this section. Lemma \ref{lemma; new 6.2}, which appears in \S\ref{section; AMU equiv} and is a modification of \cite[Lemma 6.2]{joinings}, is vital to our approach for generalizing this to the infinite volume case. 

In \S\ref{section; U-equiv}, we complete our rigidity statement by proving the following theorem.

\begin{theorem}
	Let $\hat\Up(x) = \tup(x)u_{\bsig_0}$, where $\tup$ is as in Theorem \ref{thm; rigidity; AM part main thm}. There exists a $\br$-conull set $X_2'' \subseteq X_2'$ such that for all $x \in X_2''$ and for every $v_\r \in U^-$ with $xv_\r \in X_2''$, we have \[\hat\Up(xv_\r) = \hat\Up(x)v_\r.\]
	\label{thm;rigidity;U- main thm}
\end{theorem}

\subsection{Notation}\label{section;rigidity notation}

We provide here a summary of important notation in this section for ease of reference for the reader. This is only a list of notation; full explanations are given in the following section. In particular, the reader may first skip this list and only refer to it when needed.

\subsubsection*{The constant $\ell$, the set-valued map $\Up$, the set $\tilde{X_2}$, and the functions $\up_i$.}\label{ss;ell up tildeX2} Fix $\ell \in \N$. Then $\tilde{X}_2 \subseteq X_2$ is a $U$-invariant and $\br$-conull set, and $$\Up:X_2 \to \{\text{cardinality }\ell\text{ subsets of }X_1\}$$ is a measurable map such that $$\Up(xu)=\Up(x)u$$ for all $u \in U$ and $x \in \tilde{X}_2$. The $\up_i$'s are measurable maps $\up_i:\tilde{X_2} \to X_1$ such that $$\Up(x) = \{\up_1(x),\ldots, \up_\ell(x)\}$$ for all $x \in \tilde{X_2}.$

\subsubsection*{The measure $\mu$} If $\ell \ne 1$, we assume that there exists a $\Delta(U)$-ergodic measure $\mu$ on $X = X_1 \times X_2$ such that \begin{itemize}
	\item if $Z \subseteq X_2$ is $\br$-conull, then $\bigcup\limits_{x_2 \in Z}(\Up(x_2)\times\{x_2\})$ is $\mu$-conull, and
	\item if $W \subseteq X$ is $\mu$-conull, then $\pi_2(W)$ is $\br$-conull.
\end{itemize}

\subsubsection*{The set $\mathcal{P}_{d,m}$}\label{ss;Pdm} For $d,m > 0$, define $\mathcal{P}_{d,m}$ to be the set of functions $\Theta: U \to \R$ of the form $$\Theta(\t) = \min\{|P_1(\t)|^2,\ldots, |P_m(\t)|^2\},$$ where the $P_i : U \to \R$ are polynomials of degree at most $d$.

\subsubsection*{The set $X_2'$ and the constant $\rho_0$}
The set $X_2' \subseteq \tilde{X}_2$ is a $U$-invariant and $\br$-conull set on which there exists a constant $\rho_0>0$ such that for all $x \in X_2'$, $$\big(u \in B_U(\rho_0) \text{ and } \up_i(x)=\up_j(x)u\big) \implies u = 1_G.$$ (See Lemma \ref{lem; rho0 min dist}.)

%\subsubsection*{The set $X_2'$ and the constant $\rho_0$}\label{ss;X2prime rho0} Define $f:\tilde{X_2}\to \R\cup \{+\infty\}$ by $$f(x) = \min\{|\t|:\exists i \ne j \text{ such that } \up_i(x)u_\t = \up_j(x)\}.$$ By $U$-ergodicity, there exists a $U$-invariant $\br$-conull set $X_2' \subseteq \tilde{X}_2$ and a constant $\rho_0 \in \R$ such that $$f(x) \ge \rho_0$$ for all $x \in X_2'$.

\subsubsection*{The functions $\Theta_{x,h,i}$,$\ov{\Theta}_{x,h,i}$, and $q_{i,j}$}\label{ss;theta ovtheta qij}
Define $$\Theta_{x,h,i}(\t) := \min\{1, d(\up_i(x)u_\t, \Up(xh)h\inv u_\t)^2\},$$ $$\ov{\Theta}_{x,h,i}(\t) = \min\{1, d(\up_i(xu_\t), \Up(xh)h\inv u_\t)^2\},$$ and $$q_{i,j}(\t):= \min\{1,d(\up_i(x)u_\t, \up_j(xh)h\inv u_\t)^2\}.$$

\subsubsection*{The constant $\rho$}\label{ss;rho} Let $\rho>0$ be such that if $d(xu_\t, yu_\t)<\rho$, then there exists some finite collection of polynomials $p_i$ of degree at most $n$ such that $$d(xu_\t,yu_\t) = \max_i\{|p_i(\t)|\}.$$

\subsubsection*{The constants $\eps', \xi$ and the set $K$}\label{ss;epsprime xi K} Let $0<\eps'<1/2$ and let $0<\xi<\eps'/10$. $K\subseteq X_2'$ is a compact set with $\bms(K)>1-\xi$ on which every $\up_i$ is uniformly continuous.

\subsubsection*{The constant $C$}\label{ss;C} Let $C$ be the constant from Lemma \ref{lemma; new 6.2} for the compact set $KB_U(1)$, with $d$ and $m$ chosen so that all polynomials arising from the ${\Theta}_{x,h,i}$'s and the $q_{i,j}$'s are elements of $\mathcal{P}_{d,m}$.

\subsubsection*{The constants $C', \alpha$}\label{ss;Cprime alpha} Constants such that for all balls $V\subseteq \R^{n-1}$ and all $\eps>0$, $$\lambda(\{x \in V:|q_{i,j}|<\eps\}) \le C'\left(\frac{\eps}{\sup\limits_V|q_{i,j}|}\right)^\alpha \lambda(V)$$ for all $i, j$, where $\lambda$ is the Lebesgue measure on $U \cong \R^{n-1}$. See equation (\ref{eqn;defn of Calpha good}).
	
\subsubsection*{The constants $R$, $\eps$ and $\tilde{\rho}$}\label{ss;R eps tilderho} Let $R$ be the injectivity radius of $K$. Define $$\eps = \min\left\{1,\frac{1}{20} C \eps'\min\{\rho^2, R^2\}, \frac{1}{100}C\rho_0 (C'\ell)^{-1/\alpha}\right\}$$ and $$\tilde{\rho} = 10\eps/C.$$

\subsubsection*{The constants $\eta, \beta$}\label{ss; eta beta} Let $\beta>0$ be such that for all $x,y \in K$, $$d(x,y)<\beta \implies d(\up_i(x),\up_i(y))<\min\{\eps/2, \tilde{\rho}^{1/2}\}$$ for every $1\le i \le \ell$ and $0<\eta<\eps'$ is such that for all $x \in X_2,$ $$h \in B_{AM}(\eta) \implies d(x,xh) < \min\{\eps/2, \beta ,\tilde{\rho}^{1/2}\}.$$
	
\subsubsection*{The set $P_\xi$}\label{ss; Pxi} Define $$P_\xi = \left\{x \in X_2 : \liminf\limits_{T \to \infty} \dfrac{\ps_x(\{\t \in B_U(T): xu_\t \in K\})}{\ps_x(B_U(T))} \ge 1-2\xi\right\}.$$ It is a $U$-invariant $\br$-conull set.

\subsubsection*{The set $L_{r,K}$}\label{ss;Lrk} Define $$L_{r,K} = \{x \in X_2 : \text{ there exists } t_n \to \infty \text{ such that } xa_{-t_n} \in KB_U(1) \text{ for all } n\}.$$ It is $U$-invariant and $\br$-conull.

\subsubsection*{The set $\hat{X}_h$}\label{ss;hatXh} For $h \in B_{AM}(\eta)$, $$\hat{X}_h:= P_\xi \cap P_\xi h\inv \cap X_2' \cap X_2'h\inv \cap L_{r,K}.$$ It is $U$-invariant and $\br$-conull.

\subsubsection*{The subgroups $A', M'$}\label{ss; Aprime Mprime} Let $A' \subseteq A$ and $M' \subseteq M$ be countable dense subgroups.

\subsection{$U$-equivariant implies $AM$-equivariant}\label{section; AMU equiv}

In this section, we prove Theorem \ref{thm; rigidity; AM part main thm}.

For $d,m > 0$, define $\mathcal{P}_{d,m}$ to be the set of functions $\Theta: U \to \R$ of the form $$\Theta(\t) = \min\{|P_1(\t)|^2,\ldots, |P_m(\t)|^2\},$$ where the $P_i : U \to \R$ are polynomials of degree at most $d$.

The following lemma is critical in adapting to the infinite volume setting. Roughly speaking, it says that PS measure ``sees'' the growth of polynomials: they cannot be ``small'' everywhere within the support of the PS measure, because the PS measure is `friendly' in the sense of \cite{KLW}.

\begin{lemma}
	Fix $d,m > 0$. For any compact set $\Omega$, let $\kappa=\kappa(\Omega)$ be as in Lemma \ref{lem;PS inf kappa for compact set}. Then there exists some $C=C(\Omega,d,m)>0$ satisfying the following: for every $x\in L_{r,\Omega}$ and $T > 0$ such that $xa_{-\log (T/\kappa)}\in \Omega$ and for every $\Theta \in \mathcal{P}_{d,m}$, we have \[\frac{1}{\ps_x(B_U(T))}
	\int_{B_U(T)} \Theta(\t)d\ps_x(\t) \ge C \cdot \sup\limits_{\t \in B_U(T)} \Theta(\t),\] where $L_{r,\Omega} = \{x : \text{there exists } s_n \to \infty \text{ with } xa_{-s_n} \in \Omega\}$.
	\label{lemma; new 6.2}
\end{lemma}
\begin{proof}
	Observe that for any $\kappa, T>0$, \begin{align*}
	&\frac{1}{\ps_x(B_U(T))} \int_{B_U(T)}\Theta(\t)d\ps_x(\t)\\
	&= \frac{1}{\ps_x(B_U(T))} \int_{B_U(\kappa)} \Theta((T/\kappa)\t) d\ps_x((T/\kappa)\t) \\
	&=\frac{1}{\ps_{xa_{-\log(T/\kappa)}}(B_U(\kappa))}\int_{B_U(\kappa)}\tilde{\Theta}(\t)d\ps_{xa_{-log(T/\kappa)}}(\t)
	\end{align*}
	
	Now, assume for contradiction that the claim is false. Then, by scaling the $\Theta$'s if necessary, we may assume that we have: \begin{itemize}
		\item sequences $x_i \in L_{r,\Omega}$, $s_i \to \infty$ such that $y_i = x_i a_{-\log{(s_i/\kappa)}} \in \Omega$
		\item $\tilde{\Theta}_i \in \mathcal{P}_{d,m}$ with $\sup\limits_{B_U(\kappa)}\tilde{\Theta}_i(\t)=1$
	\end{itemize} satisfying $\frac{1}{\ps_{y_i}(B_U(\kappa))} \int_{B_U(\kappa)} \tilde{\Theta}_i(\t)d\ps_{y_i}(\t) \to 0$ as $i \to \infty$.
	
	Since the $\tilde{\Theta}_i$'s are given by uniformly bounded polynomials of bounded degree, they form an equicontinuous family. Thus, by dropping to a subsequence if necessary, we may assume that there exists $y \in \Omega$ and $\tilde{\Theta} \in \mathcal{P}_{d,m}$ with $\sup\limits_{\t \in B_U(\kappa)}\tilde{\Theta}(\t)=1$ such that $y_i \to y$ and $\tilde{\Theta}_i \to \tilde{\Theta}$.
	
	Since $g \mapsto \ps_g$ is continuous by Lemma \ref{lem;new cty ps measure} and $\sup\limits_{g\in \Omega}\ps_g(B_U(\kappa))<\infty$ by Lemma \ref{lem;PS inf kappa for compact set}, we then have that $$\int_{B_U(\kappa)}\tilde{\Theta}(\t)d\ps_y(\t) = 0.$$ 
	
	Thus, $$\ps_y(B_U(\kappa)\cap \{\t:\tilde{\Theta}(\t)\ne 0\})=0.$$ Since $\ps_y(B_U(\kappa))>0$ by definition of $\kappa$, this implies $$\ps_y(B_U(\kappa)\cap\{\t:\tilde{\Theta}(\t)= 0\})>0,$$ a contradiction to Lemma \ref{lem; vars null sets, cty of g to psg}.
\end{proof}

Recall the setup: $\ell \in \N$ and $$\Up:X_2 \to \{\text{cardinality } \ell \text{ subsets of } X_1\}$$ is such that there exists a $U$-invariant $\br$-conull set $\tilde{X}_2 \subseteq X_2$ with $$\Up(xu) = \Up(x)u$$ for all $x \in \tilde{X}_2$ and $u \in U$. There are measurable maps $\up_i : \tilde{X}_2 \to X_1$ such that $$\Up(x) = \{\up_1(x), \ldots, \up_\ell(x)\}$$ for all $x \in \tilde{X}_2$.

\begin{lemma}
	There exists a $U$-invariant $\br$-conull set $X_2' \subseteq \tilde{X}_2$ and a constant $\rho_0>0$ such that for all $x \in X_2'$ and all $1 \le i,j\le \ell$, $$\big(u \in B_U(\rho_0) \text{ and } \up_i(x)=\up_j(x)u\big) \implies u = 1_G.$$ That is, there is a positive minimum distance in the $U$ direction within $\Up(x)$.
	\label{lem; rho0 min dist}
\end{lemma}
\begin{proof}
	Define $f:\tilde{X}_2 \to \R \cup\{+\infty\}$ by $$f(x) = \min \{|\t|>0:\exists 1\le i,j \le \ell \text{ such that } \up_i(x) = \up_j(x)u_{\t}\}.$$ Suppose that $\up_i(x) = \up_j(x)u_\t$ for some $i,j$ and $u_\t \in U, \t\ne 0$. By $U$-equivariance of $\Up$, for any $u \in U$, there exist $i',j'$ such that $$\up_i(x)u = \up_{i'}(xu), \text{ and } \up_j(x)u=\up_{j'}(xu).$$ Thus, $$\up_{i'}(xu) = \up_{j'}(xu)u_\t,$$ so $f(xu) \le f(x)$. Swapping the roles of $x$ and $xu$ shows that $f(x)=f(xu)$. 
	
	Hence, by the ergodicity of $\br$, there exists a $\br$-conull set $X_2' \subseteq \tilde{X}_2$ on which $f$ is constant. If $f \equiv +\infty$ on $X_2'$, then define $\rho_0 = 1$; otherwise, let $\rho_0$ be the value of $f$ on $X_2'$. It is positive by definition of $f$.
\end{proof}

Restricting to $x \in X_2'$, where $X_2'$ is as in Lemma \ref{lem; rho0 min dist}, is necessary to ensure the uniqueness of $\btau_h$ in Proposition \ref{prop;existence of tau}.

For $x \in X_2', h \in AM,$ and $1 \le i \le \ell$, define $$\Theta_{x,h,i}(\t) := \min\{1, d(\up_i(x)u_\t, \Up(xh)h\inv u_\t)^2\}$$ and $$\ov{\Theta}_{x,h,i}(\t) = \min\{1, d(\up_i(xu_\t), \Up(xh)h\inv u_\t)^2\}.$$ We will primarily work with $\Theta_{x,h,i}$, but $\ov{\Theta}_{x,h,i}$ comes into play when we use the $U$-equivariance of $\Up$. Also define $$q_{i,j}(\t):= \min\{1,d(\up_i(x)u_\t, \up_j(xh)h\inv u_\t)^2\}.$$ The main idea of the proof is to show that $q_{i,j}(\t)$ stays bounded as $|\t|\to \infty$, showing that the points $\up_i(x), \up_j(xh)h\inv$ stay in the same $U$ orbit.

The following lemma is well known by the polynomial divergence of $U$ orbits. Recall from (\ref{def; def of d}) that $$d(\Gamma x, \Gamma y):=\min\{\|g-1_G\| : g \in G, \Gamma x = \Gamma y g\},$$ where $\|\cdot\|$ denotes the max norm.

\begin{lemma}
	There exists $\rho>0$ such that if $d(xu_\t, yu_\t) < \rho$, then there exists some finite collection of polynomials $p_i$ of degree at most $n$ such that $$d(xu_\t, yu_\t) = \max\limits_i \{|p_i(\t)|\}.$$
\end{lemma}

Note that the $\Theta_{x,h,i}$'s and $q_{i,j}$'s can be controlled by polynomials in this sense, but not necessarily the $\ov{\Theta}_{x,h,i}$'s, since the $u_\t$'s are inside the $\up_i$'s here.

Let $0<\eps'<1/2$ and $0<\xi<\eps'/10$. By Lusin's theorem, there exists a compact set $K \subseteq X_2'$ with $$\bms(K)>1-\xi$$ on which every $\up_i$ is uniformly continuous. %Note that the use of $\bms$ instead of $\br$ will not cause difficulties, as we will only use $K$ to define the $\br$-conull set $P_\xi$ in (\ref{eqn; defn Pxi}), and from then on use $P_\xi$ in our proofs. 
Let $d,m>0$ be such that all of the polynomials arising from the $\Theta_{x,h,i}$'s and $q_{i,j}$'s are elements of $\mathcal{P}_{d,m}$, and let $$0<C = C(KB_U(1), d,m) \text{ as in Lemma } \ref{lemma; new 6.2}.$$

Recall that polynomials are $(C',\alpha)$ good on $\R^{n-1}$ \cite{KleinbockMargulis, KLW}: there exist constants $C',\alpha$ that depend only on the degree of the polynomial $f$ and the dimension of the space such that for all balls $V$ and all $\eps>0$, \begin{equation}
\lambda(\{x\in V :|f(x)|<\eps\}) \le C' \left(\frac{\eps}{\sup\limits_{V}|f|}\right)^\alpha \lambda(V),\label{eqn;defn of Calpha good}\end{equation} where $\lambda$ denotes the Lebesgue measure.

Choose $C'$ and $\alpha$ such that (\ref{eqn;defn of Calpha good}) holds when $f = q_{i,j}$ for all $i,j$. Let $R$ be the injectivity radius of $K$, let $\rho_0$ be as in Lemma \ref{lem; rho0 min dist}, and define \begin{equation}\label{eqn; defn eps}\eps = \min\left\{1,\frac{1}{20} C \eps'\min\{\rho^2, R^2\}, \frac{1}{100}C\rho_0 (C'\ell)^{-1/\alpha}\right\}\end{equation} and \begin{equation}\tilde{\rho} = 10\eps/C.\label{eqn; defn tilderho}\end{equation} We remark that these have been defined to achieve three things: \begin{itemize}
	\item the $\Theta_{x,h,i}$'s and $q_{i,j}$'s will be controlled by polynomials and we will stay within the injectivity radius throughout our arguments;
	\item the definition of $\eps$, together with Corollary \ref{cor;rigidity;leq 5eps}, will give a contradiction in the proof of Lemma \ref{lem;theta leq tilderho}; and
	\item $\ov\eps: = (C'\ell)^{1/\alpha}\tilde{\rho}$, which will arise in the proof of Proposition \ref{prop;existence of tau}, is less than $\frac{1}{10}\rho_0$, giving uniqueness of $\btau_h$ in that proof.
\end{itemize}

Now, let $\beta>0$ be such that for all $x,y \in K$, $$d(x,y)<\beta \implies d(\up_i(x),\up_i(y))<\min\{\eps/2, \tilde{\rho}^{1/2}\}$$ for every $1\le i \le \ell$ and let $0<\eta<\eps'$ be such that for all $x \in X_2,$ \begin{equation} h \in B_{AM}(\eta) \implies d(x,xh) < \min\{\eps/2, \beta ,\tilde{\rho}^{1/2}\}.\label{eqn;defn of eta}\end{equation}

Define \begin{equation}
\label{eqn; defn Pxi}P_\xi = \left\{x \in X_2 : \liminf\limits_{T \to \infty} \dfrac{\ps_x(\{\t \in B_U(T): xu_\t \in K\})}{\ps_x(B_U(T))} \ge 1-2\xi\right\}.\end{equation} By (\ref{eqn;reln bt br and bms}), it is a $\br$-conull set because it is $U$-invariant and has $\bms(P_\xi)=1$, \cite[Theorem 17]{rudolph}. The following lemma will allow us to control the $\Theta_{x,h,i}$'s by understanding $P_\xi$.

For $r>0$, let $$B_{AM}(r) := B(r) \cap AM,$$ where we recall that $B(r)$ denotes the ball of radius $r$ in $G$ using the max norm.

\begin{lemma}
	If $x \in P_\xi \cap P_\xi h\inv$ for $h \in B_{AM}(\eta)$, then there exists $T_1 = T_1(x,h)$ such that for all $T \ge T_1$,  $$\frac{\ps_{x}(\{\t \in B_U(T) : xu_\t, xu_\t h \in K\})}{\ps_x(B_U(T))} \ge 1- 6\xi.$$
	\label{lem;rigidity;simultaneous returns to K}
\end{lemma}
\begin{proof}
	Since $xh \in P_\xi$, there exists $T_0>0$ such that for all $T \ge T_0$, \[\frac{\ps_{xh}(\{t \in B_U(T) : xhu_\t \in K\})}{\ps_{xh}(B_U(T))} \ge 1-3\xi.\]
	
	Write $h = a_s m$. Recall that $\ps_{ya_s}(E) = e^{\delta_\Gamma s}\ps_y(a_s E a_{-s})$, $mu_\t m\inv = u_{\t m}$, and $a_s u_\t a_{-s} = u_{\t e^{-s}}$. Note also that $\ps_{ym}(E) = \ps_{y}(m E m\inv)$. Using these, we have \begin{align*}
	&\ps_{xh}(\{\t \in B_U(T):xhu_\t \in K\}) \\
	&= e^{\delta_\Gamma s} \ps_{xm}(\{\t \in B_U(e^{-s}T) : xmu_\t a_s \in K\}) \\
	&= e^{\delta_\Gamma s}\ps_x(\{\t m : \t \in B_U(e^{-s}T), xmu_\t a_s \in K\}) \\
	&\le e^{\delta_\Gamma s} \ps_x(\{\t' \in B_U((\sqrt{n-1})e^{-s}T) : xu_{\t'}  ma_s \in K\})
	\end{align*}
	Where the last line is because we are using the max norm on $U$, not the Euclidean norm. Similarly, \begin{align*}
	\ps_{xh}(B_U(T)) &= e^{\delta_\Gamma s} \ps_{xm}(B_U(e^{-s}T)) \\
	&= e^{\delta_\Gamma s} \ps_x(B_U((\sqrt{n-1})e^{-s}T))
	\end{align*}

	Putting this together, we conclude that for all $T \ge \max\{(\sqrt{n-1})e^{-s}T_0, T_0\}$, \begin{equation}
	\ps_{x}(\{\t \in B_U(T) : xu_\t h \in K\}) \ge (1-3\xi) \ps_x(B_U(T))
	\label{eqn; avg geq 1-2xi for xh}
	\end{equation}
	
	Since $x \in P_\xi$, we can choose $T_1 \ge \max\{(\sqrt{n-1})e^{-s}T_0, T_0\}$ so that for all $T \ge T_1$, \begin{equation}
	\ps_x(\{\t \in B_U(T) : xu_\t \in K\}) \ge (1-3\xi) \ps_x(B_U(T))
	\label{eqn; avg geq 1-2xi for x}
	\end{equation}
	
	Intersecting the sets on the left hand sides of equations (\ref{eqn; avg geq 1-2xi for xh}) and (\ref{eqn; avg geq 1-2xi for x}) yields the claim.
\end{proof}

Define $$L_{r,K} = \{x \in X_2 : \text{ there exists } t_n \to \infty \text{ such that } xa_{-t_n} \in KB_U(1) \text{ for all } n\}.$$ By Poincar\'e recurrence and ergodicity of $A$, $\bms(L_{r,K})=1$. It is also $U$-invariant, hence $\br$-conull by (\ref{eqn;reln bt br and bms}). Staying within this set will be necessary for our applications of Lemma \ref{lemma; new 6.2} throughout this section.

For $h \in B_{AM}(\eta)$, define \begin{equation}\hat{X}_h:= P_\xi \cap P_\xi h\inv \cap X_2' \cap X_2'h\inv \cap L_{r,K}.\label{eqn; defn hatXh}\end{equation} The set $\hat{X}_h$ is also $U$-invariant and $\br$-conull. We will show that Proposition \ref{prop;existence of tau} holds with this $\hat{X}_h$. Recall that by the definitions of $P_\xi$ and $X_2'$, if $x \in \hat{X}_h$, this means that both $x$ and $xh$ have many returns to $K$ under $U$ and $\Up$ is $U$-equivariant at both $x$ and $xh$.

\begin{corollary}
	If $x \in \hat{X}_h$ with $h \in B_{AM}(\eta)$ and $T_1$ is as in Lemma \ref{lem;rigidity;simultaneous returns to K}, then for all $T \ge T_1$ and for all $1 \le i \le \ell$, \[\frac{1}{\ps_x(B_U(T))}\int_{B_U(T)} {\Theta}_{x,h,i}(\t)d\ps_{x}(\t) \le \eps^2 +6\xi \le 5\eps.\]
	\label{cor;rigidity;leq 5eps}
\end{corollary}
\begin{proof}
	Let $D(T):= \{\t \in B_U(T) : xu_\t, xu_\t h\in K\}$. On $D(T)$, ${\Theta}_{x,h,i}(\t)<\eps^2$ because, pointwise, there exists $j(\t)$ such that$${\Theta}_{x,h,i}(\t) = \ov{\Theta}_{x,h,j(\t)}(\t)$$ (because $x, xh \in X_2'$), and it is clear for the $\ov\Theta$'s by definition of $K$ and choice of $\eta$ in (\ref{eqn;defn of eta}). On $B(T)-D(T)$, it is bounded by 1. Thus, \begin{align*}\frac{1}{\ps_x(B_U(T))} \int_{B_U(T)} \Theta_{x,h,i}(\t)d\ps_x(\t) &\le \frac{\ps_x(B_U(T)-D(T))}{\ps_x(B_U(T))} + \frac{\ps_x(D(T)) \eps^2}{\ps_x(B_U(T))}\\{}^{\text{by Lemma \ref{lem;rigidity;simultaneous returns to K}} \leadsto}
	&\le 6\xi + \eps^2 \\
	&\le 5\eps \end{align*}
\end{proof}

\begin{lemma}
	For $h \in B_{AM}(\eta)$ and $x\in \hat{X}_h$ (defined in (\ref{eqn; defn hatXh})), there exists $T_0>0$ such that for all $i$ and all $T\ge T_0$, %and all $|\t|\ge T_0$, 
	$$\sup\limits_{\t \in B_U(T)}\Theta_{x,h,i}(\t) < \tilde{\rho}.$$\label{lem;theta leq tilderho} 
\end{lemma}
\begin{proof}
	Suppose not and let $T_1$ be as in Lemma \ref{lem;rigidity;simultaneous returns to K}. Let $\kappa = \kappa(KB_U(1))$ from Lemma \ref{lemma; new 6.2}. Since $x \in L_{r,K}$, there exists $T \ge T_1$ sufficiently large so that $xa_{-\log(T/\kappa)} \in K$ and $\sup\limits_{\t \in B_U(T)} \Theta_{x,h,i}(\t) \ge \tilde{\rho}$. 
	
	Let $\Omega_{x,h,i}(\t) = \min\{\rho^2, \Theta_{x,h,i}(\t)\}$. Recall by definition of $\rho$ that this means $\Omega_{x,h,i}(\t)$ is given by polynomials in the sense of Lemma \ref{lemma; new 6.2}. Thus, we have that \[\frac{1}{\ps_x(B_U(T))} \int_{B_U(T)} \Omega_{x,h,i}(\t) d\ps_x(\t) \ge C \tilde{\rho},\] and $\Theta_{x,h,i}\ge \Omega_{x,h,i}$, so the same is true for that function. This contradicts Corollary \ref{cor;rigidity;leq 5eps} by the definition of $\tilde{\rho}=10\eps/C$.
\end{proof}

Recall the definition $${q_{i,j}}(\t) = \min\{1, d(\up_i(x)u_\t, \up_j(xh)h\inv u_\t)^2\}.$$ %By definition of $\Theta_{x,h,, for every $\t \in U$, there exists $j$ such that $\Theta_{x,h,i}(\t) = q_{i,j}(\t).$%and $$J(T,i,j):= \{t \in B_U(T) : \Theta_{x,h,i}(\t) = {q_{i,j}}(\t)\}$$

\begin{corollary}
For $h \in B_{AM}(\eta)$, $x \in \hat{X}_h$ (defined in (\ref{eqn; defn hatXh})) and $1 \le i \le \ell$, there exists $1 \le k(i)\le \ell$ and $T_1 >0$ such that for all $T \ge T_1$, $$J(T,i,k(i)):=\{\t \in B_U(T) : \Theta_{x,h,i}(\t)=q_{i,k(i)}(\t)\}$$ satisfies \begin{itemize}
	\item $\lambda(J(T,i,k(i))) \ge \frac{1}{\ell}\lambda(B_U(T))$, where $\lambda$ is the Lebesgue measure on $U$, and
	\item $\sup\limits_{\t \in J(T,i,k(i))} q_{i,k(i)}(\t) < \tilde{\rho}$.
\end{itemize}
	\label{cor; J(t,i,k(i)) stuff for qijs}
\end{corollary}
\begin{proof}
	Suppose not. Then there exists $h \in B_{AM}(\eta), x \in \hat{X}_h$, and $1\le i \le \ell$ such that for all $1 \le k(i)\le \ell$ and for all $T_1 >0$, there exists $T\ge T_1$ such that either $$\lambda(J(T,i,k(i))) < \frac{1}{\ell}\lambda(B_U(T))$$ or $$\sup\limits_{\t \in J(T,i,k(i))}q_{i,k(i)} \ge \tilde{\rho}.$$
	
	By the pigeonhole principle, there exists $k(i)$ and $T_n \to \infty$ such that for all $n$, $$\lambda(J(T_n,i,k(i)))\ge \frac{1}{\ell}\lambda(B_U(T_n)).$$ Thus, it must be that $$\sup\limits_{t \in J(T_n,i,k(i))} q_{i,k(i)}(\t) = \sup\limits_{\t\in J(T_n,i,k(i))}\Theta_{x,h,i}(\t) \ge \tilde\rho.$$ However, this contradicts Lemma \ref{lem;theta leq tilderho}.
\end{proof}

\begin{lemma}
	If $x,y \in X_i$ are such that $d(xu_\t,yu_\t)$ stays bounded for all $\t \in U$, then there exists $u \in U$ such that $$x = yu.$$ \label{lem; polynomial nondiv}
\end{lemma}
\begin{proof}
	By direct computation, we will show that if $x=yg$, then $g \in C_G(U) = U$. Write $$g = \begin{pmatrix}
	a& \b & c \\ \d^T & E & \f^T \\ h & \i & j	\end{pmatrix}$$ where $\b, \d, \f, \i \in \R^{n-1}$ are row vectors, and $E$ is a $(n-1)\times (n-1)$ matrix. By assumption, $\|u_{-\t}gu_\t - 1_G\|$ stays bounded for all $\t \in U$. We will investigate the entries of $u_{-\t}gu_\t$.
	
	The (1,1) entry is $a+\b \cdot \t + \frac{1}{2}c|\t|^2$. Since this stays bounded for all $\t$, we conclude that $$\b = \textbf{0}, \text{ and } c=0.$$
	
	The (2,1) entry is $-a\t^T + \d^T + E\t^T - \frac{1}{2}c|\t|^2 \t^T + \frac{1}{2}|\t|^2 \f^T$. Again, since this stays bounded, we conclude that $$\f = \textbf{0},  \text{ and } E=aI,$$ where $I$ denotes the $(n-1)\times(n-1)$ identity matrix.
	
	The (3,2) entry is $\frac{1}{2}|\t|^2 \b - \t E + \i + \frac{1}{2}c|\t|^2 - (\t \cdot \f)\t + j\t$. From this, we conclude that $$E= jI.$$ But combining all of our conclusions up to this point tells us that $g$ is block lower triangular with $E = aI = jI$, so $$\det(g) = 1 = a^n j = a j^n.$$ Hence $$a=j=1 \text{ and } E = I.$$
	
	With our above assumptions, the (3,2) entry is simply $\i$. The (3,1) entry simplifies to $h + (\i -\d)\cdot \t$, so $$\i = \d,$$ from which we finally conclude that $$u_{-\t}gu_\t = g,$$ completing the proof.
\end{proof}

We are now ready to prove Proposition \ref{prop;existence of tau}.

\begin{proof}[Proof of Proposition \ref{prop;existence of tau}]

If $J(T,i,k(i))$ is as in Corollary \ref{cor; J(t,i,k(i)) stuff for qijs}, then by definition of $C', \alpha$ from (\ref{eqn;defn of Calpha good}) for $f = q_{i,k(i)}$, we have for all $T>0$, $$\frac{1}{\ell}\lambda\big(B_U(T)\big) \le \lambda\big(J(T,i,k(i))\big) \le C' \tilde{\rho}^\alpha \left(\sup\limits_{\t \in B_U(T)} q_{i,k(i)}(\t)\right)^{-\alpha} \lambda\big(B_U(T)\big)$$ which yields $$\sup\limits_{\t \in B_U(T)} q_{i,k(i)}(\t) \le (C' \ell)^{1/\alpha} \tilde{\rho}$$
	
	Thus, $q_{i,k(i)}(\t)$ stays bounded for all $\t$, and so $\up_i(x)$ and $\up_{k(i)}(xh)h\inv$ are in the same $U$-orbit by Lemma \ref{lem; polynomial nondiv}. In particular, since the bound above holds at $\t = 0$, this tells us that there exists a $\btau_h(x,\up_i(x)) \in B_U(\ov\eps)$ such that $$\up_i(x)u_{\btau_h(x,\up_i(x))}h = \up_{k(i)}(xh),$$ where $\ov\eps = (C' \ell)^{1/\alpha}\tilde{\rho}$.
	
	Note that the restriction $\btau_h(x,\up_i(x)) \in B_U(\ov\eps)$ ensures that this quantity is unique, because the constants have been chosen such that $$\ov\eps < \frac{1}{10}\rho_0,$$ where $\rho_0$ is from Lemma \ref{lem; rho0 min dist}, and is the minimum distance in the $U$-direction in $\Up(x)$. Thus, if there is another element $\up_{k(i)'}(xh) \in \Up(xh)$ such that $$\up_{k(i)'}(xh) = \up_{k(i)}(xh)u_\t$$ for some $\t$, we must have that $|\t|\ge \rho_0$, hence $u_\t \not\in B_U(\ov\eps)$.

	$U$-invariance of $\btau_h$ follows from the $U$-equivariance of $\Up$ on $X_2'$: let $u_\s$ be such that $u_\t h = h u_\s$. Then $$\up_{k(i)}(xh)u_\s = (\up_i(x)u_{\btau_h(x,\up_i(x))}h) u_\s = (\up_i(x)u_\t) u_{\btau_h(x,\up_i(x))}h,$$ and $\up_{k(i)}(xh)u_\s \in \Up(xu_\t h)$ by $U$-equivariance. 
\end{proof}

Let $A' \subseteq A$ and $M' \subseteq M$ be countable dense subgroups. Recall the assumptions made at the beginning of the section: either \begin{enumerate}
	\item $\ell =1$, or
	\item there exists a $\Delta(U)$-ergodic measure $\mu$ on $X=X_1 \times X_2$ such that \begin{itemize}
		\item if $Z \subseteq X_2$ is $\br$-conull, then $\bigcup\limits_{x_2\in Z}(\Up(x_2)\times \{x_2\})$ is $\mu$-conull, and
		\item if $W \subseteq X$ is $\mu$-conull, then $\pi_2(W)$ is $\br$-conull.
	\end{itemize}
\end{enumerate}

It is in the following lemma that these assumptions are needed.

\begin{lemma}
	For every $h \in B_{A'M'}(\eta)$, there exists a $U$-invariant $\br$-conull set $W_h\subseteq X_2'$ and a constant $\btau_h \in B_U(\ov\eps)$ such that for all $x \in W_h$ and all $1 \le i \le \ell$, $\btau_h(x,\up_i(x))= \btau_h$.\label{lem; defn Wh}
\end{lemma}

%\begin{lemma}
%	For every $h \in B_{A'M'}(\eta)$, there exists a $U$-invariant $\br$-conull set $W_h\subseteq X_2'$ and a constant $\btau_h \in B_U(\ov\eps)$ such that for all $x \in W_h$ and all $1 \le i \le \ell$, $\btau_h(x,\up_i(x))= \btau_h$.\label{lem; defn Wh}
%\end{lemma}
\begin{proof}
	We first prove this in case (1), that is, we asssume that $\ell = 1$. Then the second variable in $\btau_h$ is redundant; instead, consider $\btau_h : \hat{X}_h \to U$ as simply $\btau_h(x)$, where $\hat{X}_h$ is as in equation (\ref{eqn; defn hatXh}). By Proposition \ref{prop;existence of tau}, $$\btau_h(x)=\btau_h(xu)$$ for all $x \in \hat{X}_h$, $u \in U$. Thus, by $U$-ergodicity of $\br$, there exists a $\br$-conull set $W_h\subseteq X_2'$ and a constant $\btau_h \in B_U(\ov\eps)$ such that for all $x \in W_h$, $$\btau_h(x)=\btau_h.$$ This completes the proof of the first case.\\
	
	Now, suppose we are in case (2), so such an ergodic measure $\mu$ exists. Define $$\hat{W}_h = \bigcup\limits_{x_2 \in \hat{X}_h} (\Up(x_2)\times\{x_2\}).$$ $\hat{W}_h$ is exactly the domain of $\btau_h$, and is $\mu$-conull because $\hat{X}_h$ is $\br$-conull and our assumptions on $\mu$. Thus, $\btau_h$ is defined $\mu$-a.e.\ on $X_1 \times X_2$.

	As noted in Proposition \ref{prop;existence of tau}, $\btau_h$ is $\Delta(U)$-invariant. Thus, by ergodicity of $\mu$, there exists a $U$-invariant $\mu$-conull set $\tilde{W}_h \subseteq \hat{W}_h$ and a constant $\btau_h \in B_U(\ov\eps)$ such that $$\btau_h(x,\up_i(x)) = \btau_h \text{ for all }(x,\up_i(x))\in \tilde{W}_h.$$
	
	Now, define $$W_h := \pi_2(\tilde{W}_h)\cap \hat{X}_h.$$ By the second assumption about $\mu$, it is $\br$-conull. It satisfies the desired conditions by construction.
\end{proof}

%\begin{remark}
%	\emph{In our application to joining classification in \S\ref{section; joining classification}, the assumptions in case (2) of Lemma \ref{lem; defn Wh} will be satisfied with $\mu$ an ergodic $U$-joining for $(\br_1,\br_2)$ and $$\Up(x_2):=\pi_1(\pi_2\inv(x_2)).$$ That this $\Up$ will a.e.\ take values in cardinality $\ell$ subsets of $X_1$ for some $\ell \in \N$ is not immediately clear, and this is proven in \S\ref{section; fibers finite}.}
%\end{remark}

The following lemma will allow us to drop the restriction that $|h|<\eta$.

\begin{lemma}
	For any $h \in B_{A'M'}(\eta)$, there exists a $U$-invariant $\br$-conull set $Y_h$ with the property that for every $n \in \N$, there exists $\btau_{h^n} \in U$ such that for all $x \in Y_h$, \begin{equation}\Up(x)u_{\btau_{h^n}}h^n = \Up(xh^n).\label{eqn; identity btau satisfies}\end{equation} Moreover, $\btau_{h^n} = \btau_h+e^{-s}\btau_{h^{n-1}}m\inv$ and if $x \in Y_h$, so is $xh^n$ for all $n\in \N$. Thus, $\btau_h$ is defined in a way that satisfies (\ref{eqn; identity btau satisfies}) for all $h \in A'M'$.
	\label{lem;joinings;extend tau to hn}
\end{lemma}
\begin{proof}
	Define $Y_h := \bigcap\limits_{n \in \N} W_h h^{-n}$ where $W_h$ is as in Lemma \ref{lem; defn Wh}. Observe that $\btau_{h^n}$ satisfies $$\Up(x) u_{\btau_{h^n}}{h^n}= \Up(x{h^n}).$$ Since $x{h^n} \in W_h$ as well, we can proceed by induction: $$\Up(xh^{n}) = \Up(xh)u_{\btau_{h^{n-1}}}{h^{n-1}} = \Up(x)u_{\btau_h}hu_{\btau_{h^{n-1}}}{h^{n-1}} = \Up(x)u_{\btau_h+e^{-s}u_{\btau_{h^{n-1}}}m\inv} h^{n},$$ where $h = a_s m$. This shows that $$\btau_{h^n} =\btau_h + e^{-s}\btau_{h^{n-1}} m\inv$$ extends the definition to $h^n$ if it did not already exist (i.e. if $h^n\not\in B_{AM}(\eta)$), or alternatively that if it is already defined, then $\btau_{h^n}$ satisfies this identity.
\end{proof}

\begin{lemma}
	Define $Y= \bigcap\limits_{h \in A'M'} Y_h$. For all $x \in Y$, $$\Up|_{xA'M'}: xA'M' \to \{\text{cardinality } \ell \text{ subsets of } X_1\}$$ is uniformly continuous.
	\label{lem; up ucts}
\end{lemma}
\begin{proof}
	For $\eps'>0$, define $\ov{\eps}=(C'\ell)^{1/\alpha}\tilde{\rho}$ as in the proof of Proposition \ref{prop;existence of tau}. In particular, since $\tilde{\rho}=10\eps/C$, the definition of $\eps$ in equation (\ref{eqn; defn eps}) implies that $$\ov\eps \le \frac{1}{2}\eps'\min\{\rho^2,R^2\}.$$
	
	Now, if $w,z \in xA'M'$ with $d(w,z)<\eta$, where $\eta$ is defined in (\ref{eqn;defn of eta}), then there exists $h \in B_{A'M'}(\eta)$ such that $z=wh$. Then, by Proposition \ref{prop;existence of tau}, $$\Up(z) = \Up(wh) = \Up(w)u_{\btau_h}h,$$ where $\btau_h \in B_U(\ov\eps)$. Thus, $d(\Up(z),\Up(w))$ is bounded in terms of $\eta$ and $\ov\eps$ whenever $d(w,z)<\eta$, and both $\eta$ and $\ov\eps$ are independent of $z$ and $w$.
\end{proof}

\begin{corollary}
	There exists $\tilde{\Up}:X_2 \to \{\text{cardinality }\ell \text{ subsets of } X_1\}$ such that: \begin{enumerate}
		\item $\Up(x) = \tilde{\Up}(x)$ for all $x \in Y$ (which is $\br$-conull and both $U$ and $A'M'$-invariant);
		\item $\tilde{\Up}|_{xAM}$ is continuous for every $x \in Y$; 
		\item and for all $x \in Y$ and $h \in AM$, there exists $\btau_h$ such that $$\tilde{\Up}(xh) = \tilde{\Up}(x)u_{\btau_h}h.$$ Moreover, this extension of the function $h\mapsto \btau_h$ is continuous on $AM$.
	\end{enumerate}	
	\label{cor;up cts on AM orbit}
\end{corollary}
\begin{proof}
	Since $X_1$ is complete (because $G=\SO(n,1)^\circ$ is complete and $\Gamma_1$ is closed), $\{\text{cardinality } \ell \text{ subsets of } X_1\}$ is a complete metric space with the Hausdorff metric. Thus, $\Up|_{xA'M'}$ extends continuously to $xAM$ by the uniform continuity in Lemma \ref{lem; up ucts}. Call this continuous extension $\tilde{\Up}$. Clearly, (1) and (2) are satisfied.
	
	Let $x \in Y$, $h \in AM$, and let $h_n \in A'M'$ be such that $h_n \to h$. Then we have that $$u_{\btau_{h_n}} = \tilde{\Up}(x)\inv \tilde{\Up}(xh_n)\inv h_n\inv.$$ By the continuity of $\tilde{\Up}$ on $xAM$, the right hand side converges to $$\tilde{\Up}(x)\inv\tilde{\Up}(xh)\inv h\inv,$$ which defines $\btau_h$ in a way that satisfies (3).
\end{proof}

\begin{corollary}
	For every $h \in AM - M$, there exists $\bsig_h \in U$ satisfying $\bsig_{h^n} = \bsig_h$ for all $n \in \N$ and such that $$\tup(x)u_{\bsig_h}h = \tup(xh)u_{\bsig_h}$$ for all $x \in Y$. Moreover, $\bsig$ is continuous on $AM-M$. 
	\label{cor; joinings; existence of bsig}
\end{corollary}
\begin{proof}
	Let $h \in AM-M$ and write $h=a_sm$. By assumption on $h$, $s \ne 0$, so $I-e^{-s}m\inv$ is invertible. Define $$\bsig_h:= (I-e^{-s}m\inv)\inv \btau_h.$$ It follows from the recurrence formula for $\btau_h$ in Lemma \ref{lem;joinings;extend tau to hn} that $\bsig_{h^n} = \bsig_h$, and it satisfies the desired equality for $x \in Y$ by definition. It is continuous on $AM-M$ because $h \mapsto \btau_h$ is by Corollary \ref{cor;up cts on AM orbit}, as is the inversion $a_sm \mapsto (I-e^{-s}m\inv)\inv$.
\end{proof}

\begin{lemma}
	There exists $\bsig_0 \in U$ such that for all $h \in AM$ and all $x \in Y$, $$\tup(xh)u_{\bsig_0} = \tup(x)u_{\bsig_0}h.$$ 
\end{lemma}
\begin{proof}
	We will first how prove the lemma under the assumption that there exists $x \in K \cap Y$, a sequence $n_k\to \infty$, and $a \in A'-\{1_G\}$ such that:
	\begin{itemize}
		\item $xa^{n_k}\in K \cap Y$ for all $k$, and
		\item $xa^{n_k} \to x$.
	\end{itemize}
%	Let $\{U_n\}$ be a countable basis for the topology of $X_2$. Applying Birkhoff's theorem to the sets $U_n \cap Y$ (note $\bms(Y)=1$ by $U$-invariance, (\ref{eqn;reln bt br and bms})) under the action of $a \in A' - \{1_G\}$ tells us that there exists a set $Z \subseteq Y$ with $\bms(Z)=1$ such that for every $x \in Z$, $\{xa^n:n \ge 1\}$ is dense in $Y$.
	
%	Let $x \in Z$. Then, since $x \in Y$, the density of $\{xa^n : n \ge 1\}$ implies that there exist $n_k \to \infty$ such that \begin{itemize}
%		\item $xa^{n_k} \in Y$ for all $k$, and 
%		\item $xa^{n_k} \to x$.
%	\end{itemize} Define $\bsig_0 = \bsig_a$ and recall that $\bsig_{a^n} = \bsig_a$ for all $n \in \N$ by Corollary \ref{cor; joinings; existence of bsig}. 
	
	First, consider $h \in AM-M$. Since $x, xa^{n_k} \in K \cap Y$, we have that $$\tup(xa^{n_k}h) = \tup(xa^{n_k})u_{\bsig_h}hu_{\bsig_h}\inv \to \tup(x)u_{\bsig_h}hu_{\bsig_h}\inv = \tup(xh).$$ 
	
	On the other hand, \begin{align*}
	\tup(xa^{n_k}h) &= \tup(xh)u_{\bsig_0}a^{n_k} u_{\bsig_0}\inv\\
	&= \tup(x)u_{\bsig_h}hu_{\bsig_h}u_{\bsig_0}a^{n_k}u_{\bsig_0}\inv\\
	&= (\tup(x)u_{\bsig_0}a^{n_k}u_{\bsig_0}\inv)u_{\bsig_0}(a^{-n_k}u_{\bsig_0}\inv u_{\bsig_h}h u_{\bsig_h}\inv u_{\bsig_0}a^{n_k}) u_{\bsig_0}\inv\\
	&= \tup(xa^{n_k})u_{\bsig_0}(a^{-n_k}u_{\bsig_0}\inv u_{\bsig_h}hu_{\bsig_h}\inv u_{\bsig_0} a^{n_k})u_{\bsig_0}\inv\\
	&\to \tup(x)u_{\bsig_0} h u_{\bsig_0}\inv
	\end{align*} where the convergence follows because $a \ne 1_G$ means that $a^{-n_k}ua^{n_k} \to 1_G$ for all $u \in U$. Thus, $\bsig_h = \bsig_0$ for all $h \in AM-M$.
	
	The statement then follows for all $h \in AM$ using the continuity in Corollary \ref{cor;up cts on AM orbit}.
	
	We will now show how to establish the existence of such $x$ and $n_k \to \infty$. Let $K' \subseteq K$ be a compact set consisting of density points of $K$ and satisfying $$\bms(K')>0.9\bms(K).$$ That is, for all $x \in K'$, there exists $r_x>0$ such that for all $r\le r_x$, $$\bms(xB(r)\cap K) > \frac{1}{2}\bms(xB(r)),$$ where $B(r)$ denotes the ball of radius $r$ in $G$; see \S\ref{section; prelims}.
	
	Let $\{x_n:n\in\N\}$ be a countable dense subset of $K'$. For $m,k \in \N$, define $$f_{m,k}:= \1_{x_mB(1/k)\cap K\cap Y}.$$ Let $a \in A' - \{1_G\}$. By Birkhoff's theorem applied to the family $\{f_{m,k}\}$, there exists $Z \subseteq X$ with $\bms(Z) = 1$ such that for all $x \in Z$ and all $m,k$, there exists a sequence $n_j \to \infty$ such that for every $j$, $$xa^{n_j} \in x_mB(1/k)\cap K \cap Y.$$
	
	Now, let $x \in K' \cap Y \cap Z$. By the density of $\{x_n\}$, there exists a subsequence $x_{m_j} \to x$. Then, since $x \in Z$, we can find $n_k \to \infty$ such that $$xa^{n_k} \in xB\left(\frac{1}{k}+d(x,x_{m_k})\right) \cap K \cap Y.$$ This establishes the existence of such $x, n_k$, completing the proof.
\end{proof}

\subsection{$AMU$-equivariant implies $U^-$-equivariant}\label{section; U-equiv}

In this section, we establish Theorem \ref{thm;rigidity;U- main thm}. %Existence of a $\bms$-conull set satisfying the statement of Theorem \ref{thm;rigidity;U- main thm} follows almost identically as in \cite[Theorem 6.1]{joinings}, with only slight modifications required due to the arbitrary dimension in our case. For this reason, we refer the reader to \cite{joinings} for the details, and provide here only the changes that need to be made. We will then explain how to adjust the argument to get a $\br$-conull set.

We will first show how the following proposition, which is identical to Theorem \ref{thm;rigidity;U- main thm} except for the use of $\bms$ instead of $\br$, follows from the proof of \cite[Theorem 6.1]{joinings}, with only slight modifications required due to the arbitrary dimension in our case. 

\begin{proposition}\label{prop; U- bms}
	Let $\hat{\Up}(x) = \tup(x)u_{\bsig_0},$ where $\tup$ is as in Theorem \ref{thm; rigidity; AM part main thm}. There exists a $\bms$-conull set $X_2'' \subseteq X_2'$ such that for all $x \in X_2''$ and for every $v_\r \in U^-$, we have $$\hat{\Up}(xv_\r) = \hat{\Up}(x)v_\r.$$
\end{proposition}

We refer the reader to the proof of \cite[Theorem 6.1]{joinings} for the details, and provide here only the changes that need to be made to accommodate the arbitrary dimension. 

The heart of the proof of \cite[Theorem 6.1]{joinings} is \cite[Prop.\ 6.4]{joinings}, and it is here where all changes due to dimension appear. The first difference is the time-change map $\beta_\r(\t)$. In this case, the necessary time change is given by \[\beta_\r(\t) = \dfrac{\t + \frac{1}{2}|\t|^2e^{-s}\r}{1+e^{-s}\t\cdot \r + \frac{1}{4}e^{-2s}|\t|^2|\r|^2}.\] It is chosen so that $$v_{e^{-s}\r} u_\t = u_{\beta_{\r}(\t)}g_\r$$ where $g_\r \in AMU^-$, using the notation as in equation (6.11) in the proof of \cite[Prop.\ 6.4]{joinings}. Direct computation shows that for $|\t|\le e^{s}$ and $|\r|<\eps$, we still have that \begin{align*}
|\beta_\r(\t)|&\le \dfrac{|\t|+\frac{1}{2}|\t|^2e^{-s}|\r|}{1+\frac{1}{4}e^{-2s}|\t|^2|\r|^2 - e^{-s}|\t||\r|} \\
&\le \frac{ e^{s} + \frac{1}{2}|\r|}{1-e^{-s}|\t||\r|} \\
&\le \frac{e^{s} + \frac{1}{2}\eps}{1-e^{-s/2}\eps} \\
&= e^{s}+O(\eps),
\end{align*} so the proof carries through, up to Step 4.

In Step 4, the matrix computation to prove equation (6.19) in the proof of \cite[Prop.\ 6.4]{joinings} is more cumbersome, but not fundamentally different. We provide an outline of the approach below.

For $g_s:=g_{s,i}$ in that proof, write \[g_s = \begin{pmatrix} 1 & & \\ \v^T & I & \\ \frac{1}{2}|\v|^2 & \v & 1 \end{pmatrix} \begin{pmatrix} \lambda & & \\ & C & \\ & & \lambda\inv\end{pmatrix} \begin{pmatrix} 1 & \w & \frac{1}{2}|\w|^2 \\ & I & \w^T \\ & & 1\end{pmatrix}\] for some row vectors $\v, \w \in \R^{n-1}$, $C \in M, \lambda > 0$. Multiplying it out gives \[g_s =\begin{pmatrix} \lambda & \lambda \w & \frac{1}{2}\lambda|\w|^2 \\\lambda\v^T & \lambda \v^T \w + C & \frac{1}{2}\lambda|\w|^2\v^T + C\w^T \\ \frac{1}{2}\lambda |\v|^2 & \frac{1}{2}\lambda |\v|^2\w + \v C & \frac{1}{4}|\v|^2|\w|^2+\v C\w^T + \lambda\inv\end{pmatrix}\] Writing $g_s = \begin{pmatrix} a & \q & b \\ \x^T & B & \z^T \\ c & \y & d \end{pmatrix}$, we investigate the entries of $I-u_{-\t}g_s u_\t$. We still have that \begin{equation}d(1_G, u_{-\t}g_s u_\t)=O(1) \text{ for all }\t \in B_U(e^{s}).\label{eqn;g_s bigO1}\end{equation}

The magnitude of the (3,1) entry is $$\left|\frac{1}{2}|\t|^2a-\t\cdot \x + c +\frac{1}{2}|\t|^2(\q\cdot \t)-\t B\t^T + \y\cdot\t + \frac{1}{4}|\t|^4b - \frac{1}{2}|\t|^2(\t\cdot \z) + \frac{1}{2}|\t|^2 d\right|.$$ By considering the $t^4$ term and equation (\ref{eqn;g_s bigO1}), we conclude that $$|b|=O(e^{-4s}).$$

Using this and the $t^2$ terms of the (2,1) entry $$\left|\x^T - a\t^T - \t^T(\q\cdot \t) + B\t^T -\frac{1}{2}|\t|^2b\t^T + \frac{1}{2}|\t|^2\z^T\right|,$$ we similarly conclude that $$\left|\frac{1}{2}|\z| - |\q|\right|=O(e^{-2s}).$$ Using the (3,2) entry $$\left|\frac{1}{2}|\t|^2\q - \t B + \y +\frac{1}{2}|\t|^2b\t - (\t \cdot \z)\t + d\t\right|,$$ we get $$\left|\frac{1}{2}|\q| - |\z|\right|=O(e^{-2s}).$$

Continuing in this manner, we end up with the following conclusions:
\begin{enumerate}
	\item $|b| = O(e^{-4s})$ by the (3,1) entry
	\item $|\frac{1}{2}|\z| - |\q||=O(e^{-2s})$ from the (2,1) entry
	\item $|\frac{1}{2}|\q| - |\z||=O(e^{-2s})$ by the (3,2) entry
	\item $||\q|-|\z||=O(e^{-3s})$ from the (3,1) entry
	\item $|\q| = O(e^{-2s})$ and $|\z| = O(e^{-2s})$ using the three lines above
	\item $||B|-|a||=O(e^{-s})$ from the (2,1) entry
	\item $||B|-|d||=O(e^{-s})$ from the (3,2) entry
	\item $||a|-|d|| = O(e^{-s})$ from the two lines above
\end{enumerate}
We will further show that $$d=\lambda\inv + O(e^{-s}), \lambda = 1+O(e^{-s}) \text{ and }|B-I| = O(e^{-s}).$$ From these facts, it will follow that $d(1_G, a_sg_sa_{-s})=O(e^{-s})$, completing the proof of Proposition \ref{prop; U- bms}.

Because $d(1_G,u_{-\t}g_su_\t)=O(1)$, we know that $\lambda = 1+O(1)$, so $|\lambda \v| = O(1)$ implies that $$|\v| = O(1).$$ Similarly, $$|\w|= O(1).$$ Then, from the fact that $C = I-\lambda \v^T \w + O(1)$, we conclude $$C = I+O(1).$$ Thus, \begin{equation}d = \frac{1}{4}|\v|^2|\w|^2 + \v C \w^T + \lambda\inv = \lambda\inv + O(e^{-s}).\label{eqn;matrix computation d}\end{equation}

Now, using that $a = \lambda$, $||a|-|d|| = O(e^{-s})$, and $d = \lambda\inv + O(e^{-s})$ by equation(\ref{eqn;matrix computation d}), it follows that $$|\lambda - \lambda\inv| = O(e^{-s}).$$ This in turn implies (because $\lambda = 1 + O(1)$) that $$|\lambda^2 - 1| = O(e^{-s}),$$ so $\lambda^2 = 1 + O(e^{-s})$. Using the Taylor series for $\sqrt{1+x}$, we conclude that $$\lambda = 1+O(e^{-s}).$$

Combining the above with $|B-aI| = |B-\lambda I| = O(e^{-s})$ implies $$|B-I|=O(e^{-s}),$$ as desired.

We will now explain how to slightly change the proof of \cite[Theorem 6.1]{joinings} to yield a $U$-invariant $\bms$-conull set, hence a $\br$-conull set in light of equation (\ref{eqn;reln bt br and bms}). This will establish Theorem \ref{thm;rigidity;U- main thm}.

Let $K_\eta$ be the compact subset chosen in equation (6.4) in \cite{joinings} and $\Omega_\eta$ as in equation (6.6). More specifically, $$\Omega_\eta \subseteq \{x:x^-\in \Lambda_{\operatorname{r}}(\Gamma)\}$$ is a compact set with $$\bms(\Omega_\eta)>1-\eta$$ such that there exists $T_\eta>1$ so that for every $x \in \Omega_\eta$ and $T \ge T_\eta$, \begin{equation}
\frac{1}{\ps_x(B_U(T))} \int_{B_U(T)} \1_{K_\eta}(xu_\t)d\ps_x(\t) \ge 1-2\eta.
\end{equation}

We will need to thicken $\Omega_\eta$ slightly in the $U$-direction. By \cite[Lemma 4.4]{joinings}, there exists $r_0>0$ and $R>0$ such that for all $x \in \Omega_\eta$ and all $T \ge R$, \begin{equation}
\label{eqn; defn r0}
\ps_x(B_U(T+r_0) - B_U(T-r_0))<\eta \ps_x(B_U(T)).\end{equation} 

\begin{remark}
	\emph{Despite the stated dependence in \cite[Lemma 4.4]{joinings}, $R$ is in fact independent of $x$. The apparent dependence in that statement arises from \cite[Theorem 4.1]{joinings}, which is actually weaker than the result cited from \cite{schapira}. The original proof in \cite{schapira} shows that there is no such dependence on the base point. }
\end{remark}

We will show that if $y \in \Omega_\eta B_U(r_0)$, then for all $T \ge T_\eta + r_0$, \begin{equation}\label{eqn; omegaeta fattened}
\frac{1}{\ps_y(B_U(T))}\int_{B_U(T)}\1_{K_\eta}(yu_\t) \ge 1-3\eta
\end{equation}

Equation (\ref{eqn; defn r0}) implies that \begin{align*}
\frac{1}{\ps_x(B_U(T-r_0))}-\frac{1}{\ps_x(B_U(T+r_0))} &\le \frac{\ps_x(B_U(T+r_0))-\ps_x(B_U(T-r_0))}{\ps_x(B_U(T))\ps_x(B_U(T-r_0))} \\
&< \frac{\eta}{\ps_x(B_U(T-r_0))}
\end{align*}

Suppose now that $x \in \Omega_\eta$, $u \in B_U(r_0)$ and $T \ge T_\eta+ r_0$. Then $$xuB_U(T) \subseteq xB_U(T+r_0),$$ so together with the above we conclude that \begin{equation}\frac{1}{\ps_{xu}(B_U(T))} \ge \frac{1}{\ps_x(B_U(T+r_0))}\ge \frac{1-\eta}{\ps_x(B_U(T-r_0))} \label{eqn;xu to x r0}
\end{equation} and therefore \begin{align*}
&\frac{1}{\ps_{xu}(B_U(T))} \int_{B_U(T)}\1_{K_\eta}(xuu_\t)d\ps_{xu}(\t) \\
 {}^{\text{by (\ref{eqn;xu to x r0})} \leadsto}&\ge \frac{1-\eta}{\ps_x(B_U(T-r_0))} \int_{B_U(T-r_0)} \1_{K_\eta}(xu_\t)d\ps_x(\t) \\
&\ge 1-3\eta,
\end{align*} which establishes (\ref{eqn; omegaeta fattened}).

Thus, by adjusting constants slightly, we can use $\Omega_\eta B_U(r_0)$ in place of $\Omega_\eta$ in \cite[Prop.\ 6.4]{joinings}. Now, as in the proof of \cite[Theorem 6.1]{joinings}, by Birkhoff's theorem, there exists an $A$-invariant $\bms$-conull set $Z$ such that for all $x \in Z$, \[\frac{1}{T}\int_0^T \1_{\Omega_\eta}(xa_s)ds = \bms(\Omega_\eta)>0.9.\]

Let $x \in Z$ and $u_\t \in U$. Suppose that $|e^{-s}\t|<r_0$ and that $xa_s \in \Omega_\eta$. Then $$xu_\t a_s = xa_su_{e^{-s}\t} \in \Omega_\eta B_U(r_0).$$ From this, we conclude that if $x \in ZU$, it will have infinitely many returns under $A$ to the set $\Omega_\eta B_U(r_0)$, the set which replaced $\Omega_\eta$ in \cite[Prop.\ 6.4]{joinings}. Thus, if we define $$X_2'':= ZU,$$ then the proof of \cite[Theorem 6.1]{joinings} carries through. By equation (\ref{eqn;reln bt br and bms}), $X_2''$ is $\br$-conull, so we have established Theorem \ref{thm;rigidity;U- main thm}.

\section{Non-concentration of the PS measure near varieties}\label{section; varieties}

In this section, we prove several lemmas showing that PS measure does not concentrate near varieties. This will be needed in the next section, and is a key step for extending the results from \cite{joinings} to higher dimensions. The main result is Lemma \ref{lem;int over V small, abs cty}.

In the geometrically finite case, we will need to control the PS measure of the unit ball in $U$ based at a point that may be far out in a cusp. We will use a variation Sullivan's shadow lemma for this purpose.

Suppose that $\Gamma \backslash \H^n$ is geometrically finite. For $\xi \in \Lambda_{\bp}(\Gamma)$, let $g \in G$ be such that $g^- = \xi$. For $R>0$, define \[\mathcal{H}(\xi,R)=\bigcup\limits_{s>r}gUa_{-s}K,\] where $K$ is the maximal compact subgroup $\stab_G(o)$ as in \S\ref{section; ps}. The \emph{rank} of the horoball $\mathcal{H}(\xi,r)$ is the rank of $\stab_\Gamma(\xi)$, which is a finitely generated abelian group. It is always strictly less than $2\delta_\Gamma$.

As in \cite{bowditch}, it follows from the thick-thin decomposition of the convex core that there exists a compact set $\K_0$, a constant $R_0\ge1$, and a finite set $\{\xi_1,\ldots, \xi_m\} \subseteq \Lambda_{\bp}(\Gamma)$ such that \begin{equation}
\label{eqn; thickthin}
\supp\bms \subseteq \K_0 \sqcup \left(\bigsqcup\limits_{i=1}^m \Gamma \backslash \Gamma \mathcal{H}(\xi_i, R_0)\right)
\end{equation}

The following version of Sullivan's shadow lemma is due to Maucourant and Schapira, \cite{MacShap}:

\begin{lemma}\cite[Lemma 5.1, Remark 5.2]{MacShap}
	There exists a constant $R \ge 1$ such that for all $x \in \supp\bms$, \[R\inv T^{\delta_\Gamma}e^{(r - \delta_\Gamma)d(xa_{-\log T}, \K_0)} \le \ps_x(B_U(T)) \le R T^{\delta_\Gamma}e^{(r - \delta_\Gamma)d(xa_{-\log T}, \K_0)},\] where $r$ is the rank of the cusp containing $xa_{-\log T}$, and is zero if $xa_{-\log T} \in \K_0$.
	\label{lem; shadow lem}
\end{lemma} 

Let \begin{equation}N_{1/2}(\supp\bms):=(\supp\bms)B_U(1/2). \label{defn of N1suppbms}
\end{equation}  We will need to extend Lemma \ref{lem; shadow lem} to the following:

\begin{corollary}
	Suppose that all cusps have rank $n-1$. There exists a constant $R>0$ such that for all $x \in N_{1/2}(\supp\bms)$, and all $T>1$ we have
\[R\inv T^{\delta_\Gamma}e^{(r - \delta_\Gamma)d(xa_{-\log T}, \K_0)} \le \ps_x(B_U(T)) \le R T^{\delta_\Gamma}e^{(r - \delta_\Gamma)d(xa_{-\log T}, \K_0)},\]where $r$ is the rank of the cusp containing $xa_{-\log T}$, and is zero if $xa_{-\log T} \in \K_0$.
\label{cor; extended shadow lem}
\end{corollary}
\begin{proof}
	Let $x \in N_{1/2}(\supp\bms)$. By definition of $N_{1/2}(\supp\bms)$, there exists $$x' \in xB_U(1)\cap \supp\bms.$$ Thus, $$x'B_U(T-1) \subseteq xB_U(T) \subseteq x'B_U(T+1),$$ and so by Lemma \ref{lem; shadow lem} and since $T>1$, there exists $R_0>0$ such that \begin{align*}&R_0\inv (T/2)^{\delta_\Gamma}e^{(r' - \delta_\Gamma)d(x'a_{-\log T}, \K_0)} \\
	&\le \ps_x(B_U(T)) \\
	&\le R_0 (2T)^{\delta_\Gamma}e^{(r' - \delta_\Gamma)d(x'a_{-\log T}, \K_0)},\end{align*} where $r'$ is the rank of the cusp containing $x'a_{-\log T}.$ 
	
	Note that \begin{equation}
d(xa_{-\log T},\K_0) -  T\inv \le d(x'a_{-\log T},\K_0)\le d(xa_{-\log T},\K_0) +  T\inv
\label{eqn; kappatinv ineq}
	\end{equation}
	
	We will first consider the upper bound by cases. Suppose that $r'=0$. Then the upper bound yields \begin{align*}
	\ps_x(B_U(T))&\le R_0(2T)^{\delta_\Gamma} \le R_0(2T)^{{\delta_\Gamma}}e^{(r-{\delta_\Gamma})d(xa_{-\log T},\K_0)}e^{(n-1-\delta_\Gamma)},
	\end{align*} where the last inequality follows because if $r = 0$, $e^{(r-{\delta_\Gamma})d(xa_{-\log T},\K_0)}=1$, and otherwise, $e^{(r-{\delta_\Gamma})d(xa_{-\log T},\K_0)}>1$, and also $e^{(n-1-\delta_\Gamma)}\ge1.$
	
	Now, suppose instead that $r' = n-1$. Then by using (\ref{eqn; kappatinv ineq}), we obtain \begin{align*}\ps_x(B_U(T)) &\le R_0(2T)^{\delta_\Gamma} e^{(r'-\delta_\Gamma)d(xa_{-\log T},\K_0)}e^{(n-1-\delta_\Gamma)}\\
	&= R_0(2T)^{\delta_\Gamma} e^{(r-\delta_\Gamma)d(xa_{-\log T},\K_0)}e^{(n-1-\delta_\Gamma)} \end{align*} where the last equality is because either $r = n-1= r'$, or $r=0$, in which case $d(xa_{-\log T},\K_0)=0$. 
	
	We will now consider the lower bound by cases. Again, first suppose that $r'=0$. Then  $d(x'a_{-\log T},\K_0)=0$, so we have \begin{align*}
	\ps_x(B_U(T))&\ge R_0\inv (T/2)^{\delta_\Gamma}e^{(r'-\delta_\Gamma)d(x'a_{-\log T,}\K_0)} \\
		&= R_0\inv (T/2)^{\delta_\Gamma}e^{(r-\delta_\Gamma)d(x'a_{-\log T,}\K_0)}\\
		&\ge R_0\inv (T/2)^{\delta_\Gamma}e^{(r-\delta_\Gamma)d(xa_{-\log T,}\K_0)}e^{-(n-1-\delta_\Gamma)},\end{align*} where the last line follows from (\ref{eqn; kappatinv ineq}) if $r=n-1$, and if $r=0$, then $$e^{(r-\delta_\Gamma)d(xa_{-\log T,}\K_0)}=e^{(r'-\delta_\Gamma)d(xa_{-\log T,}\K_0)}=1$$ and $e^{-(n-1-\delta_\Gamma)} <1$.
		
	Thus, letting $R=R_0 2^{\delta_\Gamma}e^{(n-1-\delta_\Gamma)}$ establishes the claim.	
\end{proof}

\begin{corollary}
	Suppose that all cusps have rank $n-1$, and let $R$ be as in Corollary \ref{cor; extended shadow lem}. Then for every $y \in N_{1/2}(\supp\bms)$ and every $\eps>0$, we have
	
	%Let $K \subseteq N_1(\supp\bms)\cap \{x:x^-\in\Lambda_{\operatorname{r}}(\Gamma)\}$ be compact and let $\kappa = \kappa(K)$ as in Lemma \ref{lem;PS inf kappa for compact set}. 
	\begin{enumerate}
		\item $\ps_y(B_U(\eps)) \le R^2 \eps^{2\delta_\Gamma - n+1}\ps_y(B_U(1))$ if $\eps<1$,
		\item $\ps_y(B_U(\eps))\le R^2 \eps^{2\delta_\Gamma}\ps_y(B_U(1))$ if $\eps \ge1$.
	\end{enumerate}\label{cor; cor to shad}
\end{corollary}
\begin{proof}
	First, note that by assumption on the rank of the cusps, $$\delta_\Gamma > (n-1)/2.$$	

	By Corollary \ref{cor; extended shadow lem}, we have that \begin{equation}
	\ps_y(B_U(1)) \ge R\inv e^{(r - {\delta_\Gamma})d(y,\K_0)}\label{low bd bu1}
	\end{equation} where $r$ is the rank of the cusp containing $y$, and is zero if $y\in\K_0$. Similarly, if $r_\eps$ denotes the rank of the cusp containing $ya_{-\log \eps}$, then
	\begin{equation}\label{up bd bueps}
	\ps_y(B_U(\eps)) \le R\eps^{\delta_\Gamma}e^{(r_\eps-{\delta_\Gamma})d(ya_{-\log\eps},\K_0)}.
	\end{equation}
	We also have \begin{equation}
	d(y,\K_0)-|\log \eps|\le d(ya_{-\log \eps},\K_0) \le d(y,\K_0)+|\log\eps|. \label{dlogeps log kap}
	\end{equation}
	Let $0<\eps<1$ and assume first that $r_\eps = n-1$, so that $r_\eps-{\delta_\Gamma}\ge 0$. Then by (\ref{up bd bueps}) and (\ref{dlogeps log kap}), we have \begin{align*}
	\ps_y(B_U(\eps)) & \le R\eps^{\delta_\Gamma}e^{(n-1 - {\delta_\Gamma})d(y,\K_0)}\eps^{({\delta_\Gamma}-(n-1))} \\
	{}^{\text{by (\ref{low bd bu1})} \leadsto}&\le \ps_y(B_U(1))R^2 \eps^{2{\delta_\Gamma}-(n-1)}e^{(n-1-r)d(y,\K_0)}\\ 
	&\le \ps_y(B_U(1))R^2 \eps^{2{\delta_\Gamma}-(n-1)}\\
	\end{align*} where the last line follows because if $r =0$, then $d(y,\K_0)=0$, and otherwise, $r_\eps = r = n-1.$
	
	Now, suppose that $r_\eps = 0$. Then we have by (\ref{up bd bueps}),
	\begin{align*}
	\ps_y(B_U(\eps))&\le R\eps^{\delta_\Gamma}e^{-\delta_\Gamma d(ya_{-\log \eps},\K_0)}\\
	&\le R\eps^{\delta_\Gamma}\\
{}^{\text{by (\ref{low bd bu1})} \leadsto}&\le R^2\eps^{\delta_\Gamma}e^{-rd(y,\K_0)}\ps_y(B_U(1)) \\
	&\le R^2\eps^{\delta_\Gamma}\ps_y(B_U(1))\\
	&\le R^2 \eps^{2\delta_\Gamma-(n-1)}\ps_y(B_U(1)),
	\end{align*} where the last line follows because $\delta_\Gamma \ge 2\delta_\Gamma -(n-1)$ and $\eps<1$. This establishes the first case.
	
	Now, assume that $\eps\ge 1$, so $\log\eps\ge 0$. We again consider cases. First, suppose that $r_\eps = n-1$, Then by (\ref{up bd bueps}) and (\ref{dlogeps log kap}), we have \begin{align*}
	\ps_y(B_U(\eps))&\le R \eps^{\delta_\Gamma}e^{(n-1-\delta_\Gamma)d(y,\K_0)}\eps^{n-1-\delta_\Gamma} \\
	{}^{\text{by (\ref{low bd bu1})} \leadsto}&\le R^2 \eps^{n-1}e^{(n-1-r)d(y,\K_0)}\ps_y(B_U(1)) \\
	&\le R^2 \eps^{n-1}\ps_y(B_U(1))\\
	&\le R^2 \eps^{2\delta_\Gamma}\ps_y(B_U(1))
	\end{align*} where the second to last line follows because $e^{(n-1-r)d(y,\K_0)}=0$ when $r \in \{0,n-1\}$, and the final line because $\delta_\Gamma > (n-1)/2$.
	
	Now, suppose that $r_\eps = 0$. Then again by  (\ref{up bd bueps}) and (\ref{dlogeps log kap}), we have \begin{align*}
	\ps_y(B_U(\eps))&\le R\eps^{\delta_\Gamma}e^{-\delta_\Gamma d(y,\K_0)}\eps^{\delta_\Gamma} \\
	{}^{\text{by (\ref{low bd bu1})} \leadsto}&\le R^2\eps^{2\delta_\Gamma} e^{-r d(y,\K_0)} \ps_y(B_U(1))\\
	&\le R^2\eps^{2\delta_\Gamma}\ps_y(B_U(1)),
	\end{align*} which completes the second case.
\end{proof}

For $d \in \N$ and $c>0$, define \begin{align*}\F_{d,m} = &\{f:B_U(1)\to U : f=(f_1,\ldots,f_{n-1}) \text{ with every } f_i \text{ a polynomial }\\&\text{of degree at most } d, \text{ and all coefficients of each } f_i \in [m\inv, m]\}\end{align*} Note that it is a compact subset of $C(B_U(1))$. For $f \in \F_{d,m}$ and $r>0$, define $$N_r(f) := \{\t \in B_U(1) : |f(\t)|< r\}.$$ %For $\kappa>0$, define $$N_{r}^\kappa(f)=N_r(f) \cap B_U(\kappa).$$

\begin{lemma}
	Assume that all cusps have rank $n-1$ and that $\delta_\Gamma > n-\frac{5}{4}$. Let $d,m>0$. Then there exist constants $\ov{c}>0$ and $\alpha>0$ such that for every $y \in N_{1/2}(\supp\bms)$, every $0<\eps<1/4$, and every $f \in \F_{2,m}$, $$\ps_y(N_\eps(f)) < \ov{c}\eps^\alpha \ps_y(B_U(1)).$$
	\label{lem; geomfin var nbhd small}
\end{lemma}
\begin{proof}
	By Corollary \ref{cor; cor to shad}, $\alpha_1:=  2\delta_\Gamma -(n-1)$ is such that for all $y \in N_{1/2}(\supp\bms)$, \begin{equation}
\ps_y(B_U(\eps)) < R^2\eps^{\alpha_1} \ps_y(B_U(1)). \end{equation} Let $\{z_1, \ldots, z_k\}$ be a maximal $\eps$-separated set in $N_\eps(f)$. 
	
	\textbf{Claim:} There exists some constant $d'>0$ such that $k \le d' \eps^{\frac{3-2n}{2}}.$
	\begin{proof}[Proof of claim]
		By the mean value theorem, there exists $c'$ such that for all $\s,\t \in B_U(1)$, $$|f(\t + \eps \s)| \le |f(\t)| + c' \eps.$$ Thus, since $z_i \in N_\eps(f)$, for all $ 1 \le i \le k$ and for all $\t \in z_i B_U(\eps)$, we have $$|f(\t)| < (1+c')\eps.$$ Hence, we have $$\bigsqcup\limits_{i=1}^k z_i B_U(\eps/4) \subseteq N_{(1+c')\eps}(f).$$ 	Thus, there exists some constant $d>0$ such that $dk(\eps/4)^{n-1} \le \lambda(N_{(1+c')\eps}(f))$, where $\lambda$ denotes the Lebesgue measure.
		
		%By \cite{Brudnyi}, there exists a constant $c''$ such that $$\lambda(N_{(1+c')\eps}(f)) < c'' \eps^{1/2},$$ so we have that there is some constant $d'>0$ such that $$k \le d' \eps^{\frac{3-2n}{2}},$$ which establishes the claim.
	\end{proof}

Because $\{z_1,\ldots, z_m\}$ is a maximal $\eps$-separated subset of $N_\eps(f)$, we have that $$N_\eps(f) \subseteq \bigcup\limits_{i=1}^k z_i B_U(2\eps).$$

By Corollary \ref{cor; cor to shad}, there exists $d_1>0$ such that for all $w \in N_{1/2}(\supp\bms)$, $$\ps_w(B_U(2\eps)) \le R^2d_1 \eps^{\alpha_1}\ps_w(B_U(1)).$$ Then we also have that $$\ps_{z_i}(B_U(2\eps)) \le R^2d_1 \eps^{\alpha_1}\ps_w(B_U(1)).$$ This is because if $z_i \not\in N_{1/2}(\supp\bms)$, then $z_iB_U(2\eps)\cap \supp\bms = \emptyset$, and so the left hand side is zero. 

Also by Corollary \ref{cor; cor to shad}, there exists $d_2>0$ such that for all $w \in N_{1/2}(\supp\bms),$  $$\ps_w(B_U(2)) \le R^2d_2 \ps_{w}(B_U(1)).$$ From this, we obtain

%\textbf{Claim 2:} Let $d_1> 0, d_2\ge 1$ be such that for all $w \in N_{1/2}(\supp\bms)$, $$\ps_w(B_U(4\eps)) < R^2d_1 \eps^{\alpha_1}\ps_w(B_U(1))$$ and $$\ps_w(B_U(2)) < R^2d_2 \ps_{w}(B_U(1)).$$ (Such constants exists by Corollary \ref{cor; cor to shad}.) Then for all $1 \le i \le k$, $$\ps_{z_i}(B_U(2\eps)) \le R^4d_1d_2 \eps^{\alpha_1} \ps_{z_i}(B_U(1)).$$
%\begin{proof}[Proof of claim 2]
%	\tcr{Check this again: seems changing to $N_{1/2}(\supp\bms)$ made this claim trivial.}
	
%	If $z_i \in N_{1/2}(\supp\bms)$, this follows immediately by choice of $d_1$, since $d_2,R \ge 1$.

%	Now, suppose that $z_i \not\in N_{1/2}(\supp\bms)$. First, if $z_i B_U(2\eps)\cap \supp\bms = \emptyset$, then $\ps_{z_i}(B_U(2\eps)) = 0$ and there is nothing to prove, so assume not.
	
%	Let $y_i \in z_iB_U(2\eps)\cap \supp\bms$. Then we have $$z_iB_U(2\eps)\subseteq y_iB_U(4\eps)$$ and $$y_iB_U(1) \subseteq z_i B_U(2).$$ Note that $y_iU=z_iU$, so their PS measures are the same. We have \begin{align*}
%	\ps_{z_iU}(z_iB_U(2\eps)) &\le \ps_{z_i U}(y_iB_U(4\eps)) \\
%	&\le R^2d_1 \eps^{\alpha_1}\ps_{z_iU}(y_iB_U(1))\\
%	&\le R^2d_1 \eps^{\alpha_1}\ps_{z_iU}(z_iB_U(2))\\
%	&\le R^4d_1d_2 \eps^{\alpha_1}\ps_{z_iU}(z_iB_U(1)),
%	\end{align*} 	which completes the proof of the claim.
%\end{proof}
\begin{align*}
\ps_y(N_\eps(f)) &\le \sum\limits_{i=1}^k \ps_{z_i}(B_U(2\eps)) \\
&\le \sum\limits_{i=1}^k R^2d_1\eps^{\alpha_1} \ps_{z_i}(B_U(1)) \\
&\le \sum\limits_{i=1}^k R^2d_1 \eps^{\alpha_1}\ps_y(B_U(2)) &\text{ because } z_iB_U(1) \subseteq yB_U(2)\\
&\le R^4 kd_1 d_2 \eps^{\alpha_1}\ps_y(B_U(1)) &\text{ by definition of } d_2\\
&\le \ov{c} \eps^{\alpha_1}\eps^{\frac{3-2n}{2}}\ps_y(B_U(1)),
\end{align*} where $\ov{c} =R^4 d'd_1d_2$. Let $\alpha = \frac{3-2n+2\alpha_1}{2}$. Since $\alpha_1 = 2\delta_\Gamma -(n-1)$, the assumption $\delta_\Gamma > n-\frac{5}{4}$ ensures $\alpha>0$.
\end{proof}

\begin{lemma}
	Let $K \subseteq N_{1/2}(\supp\bms)$ be compact and let $d,m>0$. Then for every $\eta>0$, there exists $\eps>0$ such that for all $y \in K$ and for all $f \in \F_{d,m}$, $$\ps_y(N_\eps(f)) < \eta \ps_y(B_U(1)).$$
	\label{lem;var nbhd small}
\end{lemma}
\begin{proof}
	Suppose not. Then there exists $\eta>0$ and sequences $f_i \in \F_{d,m}$, $y_i \in K$, and $\eps_i \to 0$ such that $$\ps_{y_i}(N_{\eps_i}(f_i)) \ge \eta \ps_{y_i}(B_U(1)).$$ By compactness of $K$ and of $\F_{d,m}$, we may assume that there exists $f \in \F_{d,m}$ and $y_\infty \in K$ such that $f\to f_i$ uniformly and $y_i \to y_\infty$.
	
	Let $V = \{t \in B_U(1): f(\t) = 0\}$. Since $f_i \to f$ uniformly, for each $i$ there exists $\eps_i'>0$ such that $$N_{\eps_i}(f_i) \subseteq N_{\eps_i'}(f)$$ and $\eps_i'\to 0$. Thus, we have that for all $i$, $$\ps_{y_i}(N_{\eps_i'}(f))\ge \eta \ps_{y_i}(B_U(1)).$$ By the continuity of $g \mapsto \ps_g$ in Lemma \ref{lem;new cty ps measure}, it follows that $$\ps_{y_\infty}(V) \ge \eta \ps_{y_\infty}(B_U(1)).$$ However, $\ps_{y_\infty}(B_U(1))>0$ because $y_\infty \in N_{1/2}(\supp\bms)$, and $\ps_{y_\infty}(V)=0$ by Lemma \ref{lem; vars null sets, cty of g to psg}, so this is a contradiction.
\end{proof}

%We will use Lemma \ref{lem;var nbhd small} in the context of $\Gamma$ convex cocompact and $K=\supp\bms$, as follows.

We will use Lemma \ref{lem;var nbhd small} in the context of $\Gamma$ convex cocompact and $K=N_{1/2}(\supp\bms)$, as follows.

%\tcr{NEW:}
%\begin{lemma} Assume that $\Gamma$ is either \begin{itemize}
%		\item convex cocompact, or
%		\item geometrically finite with all cusps of rank $n-1$ and $\delta_\Gamma > n-\frac{5}{4}$.
%	\end{itemize} Let $K \subseteq N_1(\supp\bms)\cap\{x:x^- \in \Lambda_{\operatorname{r}}(\Gamma)\}$ be compact, and let $\kappa = \kappa(K)$ as in Lemma \ref{lem;PS inf kappa for compact set}. Let $g \in \F_{d,m}$ for some $d,m$. For every $f \in C_c(\Gamma \backslash G)$ and for every $\eta>0$, there exists $\eps>0$ and $T_0 = T_0(f,\eps)>0$ such that for all $T\ge T_0$ and all $y \in K$, \begin{align*}
%	e^{-\delta_\Gamma s} \int_{a_{-s}N_{\eps'}^\kappa({g})a_s} f(yu_\t)d\t &= e^{(n-1-\delta_\Gamma)s}\int_{N_{\eps'}^\kappa({g})} f(ya_{-s}u_\t a_s)d\t \\
%	&\ll_f \eta\ps_{ya_{-s}}(B_U(\kappa)),
%	\end{align*}%\[e^{-\delta_\Gamma s} \int_{a_{-s}N_{\eps'}({g})a_s} f(yu_\t)d\t = e^{(n-1-\delta_\Gamma)s}\int_{N_{\eps'}({g})} f(ya_{-s}u_\t a_s)d\t \ll_f \eta\ps_{ya_{-s}}(B_U(1)),\] 
%	where $s = \log (T/\kappa)$. ($\ll_f \eta$ means $\le k \eta$ for some constant $k$ that depends only on $f$.)
%	\label{lem;int over V small, abs cty}
%\end{lemma}

\begin{lemma} Assume that $\Gamma$ is either \begin{itemize}
		\item convex cocompact, or
		\item geometrically finite with all cusps of rank $n-1$ and $\delta_\Gamma > n-\frac{5}{4}$.
	\end{itemize} Let $g \in \F_{d,m}$ for some $d,m$. For every $f \in C_c(\Gamma \backslash G)$ and for every $\eta>0$, there exists $\eps>0$ and $T_0 = T_0(f,\eps)>0$ such that for all $T\ge T_0$, if $ya_{-s} \in K:=N_{1/2}(\supp\bms)\cap\{x:x^-\in\Lambda_{\operatorname{r}}(\Gamma)\}$, where $s= \log T$, we have \begin{align*}
	e^{-\delta_\Gamma s} \int_{a_{-s}N_{\eps'}({g})a_s} f(yu_\t)d\t &= e^{(n-1-\delta_\Gamma)s}\int_{N_{\eps'}({g})} f(ya_{-s}u_\t a_s)d\t \\
	&\ll_f \eta\ps_{ya_{-s}}(B_U(1)).
\end{align*}%\[e^{-\delta_\Gamma s} \int_{a_{-s}N_{\eps'}({g})a_s} f(yu_\t)d\t = e^{(n-1-\delta_\Gamma)s}\int_{N_{\eps'}({g})} f(ya_{-s}u_\t a_s)d\t \ll_f \eta\ps_{ya_{-s}}(B_U(1)),\] 
%where $s = \log T$. 
($\ll_f \eta$ means $\le k \eta$ for some constant $k$ that depends only on $f$.)
	\label{lem;int over V small, abs cty}
\end{lemma}

%\begin{lemma}
%	Let $V \subseteq U$ be a proper subvariety. For every $f \in C_c(X_1)$ and for every $\eta>0$, there exists $\eps'>0$ and $T_0 = T_0(f,\eps')>0$ such that for all $T \ge T_0$ and all $y \in \supp\bms$,\[e^{-\delta_\Gamma s} \int_{a_{-s}N_{\eps'}({V})a_s} f(yu_\t)d\t = e^{(n-1-\delta_\Gamma)s}\int_{N_{\eps'}({V})} f(ya_{-s}u_\t a_s)d\t \ll_f \eta,\] where $s = \log T$. ($\ll_f \eta$ means $\le k \eta$ for some constant $k$ that depends only on $f$.)
%	\label{lem;int over V small, abs cty}
%\end{lemma}

To prove Lemma \ref{lem;int over V small, abs cty}, we need a fact from \cite{joinings}, which requires the following definition. Let $$P = AMU^-$$ and let $P_r$ denote the ball of radius $r$ in $P$. For $\eps_0, \eps_1 > 0$, we say that $zP_{\eps_1} B_U(\eps_0)$ is an \emph{admissible box} if it is the injective image of $P_{\eps_1}B_U(\eps_0)$ in $\Gamma \backslash G$ under the map $g\mapsto zg$, and $\ps_{zp}(zpB_U(\eps_0)) \ne 0$ for all $p \in P_{\eps_1}$. In the statement of the next lemma, we will assume our functions are supported within an admissible box. By a partition of unity argument, there is no loss of generality by making this assumption.

\begin{lemma}\cite[Claim A in Theorem 4.6]{joinings} \label{fact; claim A}
	Let $\xi \in C_c(X_1)$ be supported within the admissible box $zP_{\eps_1}B_U(\eps_0)$ with $0<\eps_1 \le \eps_0$. Suppose that $x^- \in \Lambda_{\operatorname{r}}(\Gamma)$, and let $\Omega$ be a compact set such that there exists $t_n \to \infty$ with $xa_{-t_n} \in \Omega$ for all $n$. Let $s_0$ be such that $$x_0 := xa_{-s_0} \in \Omega.$$ For $\rho>0$ and every $y \in x_0U$, suppose that $f_y \in C(yB_U(\rho))$ is such that \[0\le f_y \le 1 \text{ and } f=1 \text{ on } yB_U(\rho/8).\] Then there exists $c>0$ depending only on $\supp\xi$ (in particular, $\operatorname{diam}(\supp\xi)$ is a possible choice) such that for all $y \in x_0 U$, \[e^{(n-1-\delta_\Gamma)s_0}\int_{U}\xi (yu_\t a_{s_0})f_y(yu_\t)d\t \ll_\xi \ps_y(f_{y,ce^{-s_0}\eps_1,+})\] %$c>1$ \tcr{do we mean $c>0$?} depending only on the injectivity radii of $\supp \xi$ and $\supp f_y$ (hence, one uniform $c$ in the convex cocompact case) such that for all $y \in x_0 U$, \[e^{(n-1-\delta_\Gamma)s_0}\int_{U}\xi (yu_\t a_{s_0})f_y(yu_\t)d\t \ll_\phi \ps_y(f_{y,ce^{-s_0}\eps_1,+})\]
\end{lemma}

\begin{proof}[Proof of Lemma \ref{lem;int over V small, abs cty}]
	
	Let $f \in C_c(X_1)$. By a partition of unity argument, we may assume that $\supp(f)$ is contained in some admissible box $zP_{\eps_1}U_{\eps_0}$.
	
	Fix $\eta>0$. By Lemma \ref{lem; geomfin var nbhd small} (in the geometrically finite case) or \ref{lem;var nbhd small} (in the convex cocompact case), there exists $\eps>0$ such that for all $w \in N_{1/2}(\supp\bms)$, we have \begin{equation}\label{eqn; sec 4 def eps}\ps_w(N_{2\eps}({g}))<\eta\ps_w(B_U(1)).\end{equation} Let $c>0$ be as from Lemma \ref{fact; claim A} above. Let $T_0 = T_0(f,\eps)>0$ be such that $$ce^{-s_0}\eps_1 < \eps/20,$$ where $s_0 = \log T_0$. 
	
	%Fix $\eta>0$. By Lemma \ref{lem;var nbhd small}, there exists $\eps'>0$ such that for all $w \in \supp\bms$, we have $$\ps_w(N_{2\eps'}({V}))<\eta.$$ Let $c>1$ be as from Lemma \ref{fact; claim A} above. Let $T_0 = T_0(f,\eps')>0$ be such that $$ce^{-s_0}\eps_1 < \eps'/20,$$ where $s_0 = \log T_0$. 
	
	Let $y\in K$, let $T \ge T_0$, and define $s = \log T$. Let $I_T$ be a maximal set of points in $ya_{-s}N_{\eps}({g})$ such that the balls $\{zB_U(\eps/16):z\in I_T\}$ are disjoint. Thus, $$\{zB_U(\eps/4):z \in I_T\}$$ covers $ya_{-s}N_{\eps}({g})$. Let $\{f_z : z \in I_T\}$ be a partition of unity subordinate to this cover. Then we have: \begin{align*}
	e^{(n-1-\delta_\Gamma)s} \int_{N_{\eps}({g})} f(ya_{-s} u_\t a_{s})d\t &\le e^{(n-1-\delta_\Gamma)s} \sum\limits_{z \in I_T} \int_{B_U(\eps/4)} f(zu_\t a_{s})f_z(zu_\t)d\t \\
{}^{\text{by Lemma \ref{fact; claim A}} \leadsto}	& \ll_f\sum\limits_{z\in I_T} \ps_z(f_{z,\eps/20,+}) \\
	&\ll_f \kappa \ps_{ya_{-s}}(N_{2\eps}({g}))\\
{}^{\text{by }(\ref{eqn; sec 4 def eps}) \leadsto}	&\ll_f \kappa \eta \ps_{ya_{-s}}(B_U(1)) 
	\end{align*} where $\kappa$ is the multiplicity of the cover given by the Besicovitch covering theorem; $\kappa$ depends only on the dimension $n$.
\end{proof}

\section{Joinings}\label{section; joining classification}

In this section, we prove Theorem \ref{thm; final classification}. In particular, throughout this section, we assume either that the $\Gamma_i$'s are convex cocompact, or that they are geometrically finite with all cusps of rank $n-1$ and critical exponents $\delta_{\Gamma_i}>n-\frac{5}{4}.$

Let $\mu$ be an ergodic $U$-joining for $(\br_1, \br_2)$. In \S\ref{section; fibers finite}, we prove that it must be the case that $\br$-a.e.\ fiber of $\pi_2$, the projection map from $X \to X_2$, is finite, as otherwise the joining measure $\mu$ would be invariant under a nontrivial connected subgroup of $U \times \{1_G\}$, which is impossible by \cite[Lemma 7.16]{joinings}. More precisely, in \S\ref{section; fibers finite}, we prove:

\begin{theorem}[c.f. \cite{joinings}, Theorem 7.17]
	There exists a positive integer $\ell >0$ and a $\br$-conull subset $\tilde{X}_2 \subseteq X_2$ so that $\pi_2\inv(x_2)$ has cardinality $\ell$ for every $x_2 \in \tilde{X}_2$. Moreover, the fiber measures $\mu_{x_2}^{\pi_2}$ are uniform measures for each $x_2 \in \tilde{X}_2$.
	\label{thm; 717}
\end{theorem} 

This will allow us to reduce to considering $U$-equivariant set-valued maps, which we proved rigidity for in \S\ref{section; rigidity}. Specifically, we will define for $x \in \tilde{X}_2$ from the previous theorem $$\Up(x):=\pi_1(\pi_2\inv(x_2)).$$ In \S\ref{section; final joining classification}, we prove the following more precise formulation of Theorem \ref{thm; final classification}, assuming Theorem \ref{thm; 717}: %using Theorems \ref{thm; rigidity; AM part main thm} and \ref{thm;rigidity;U- main thm}, as in the proof of \cite[Prop. 7.23]{joinings}:

\begin{theorem}\label{thm; finite cover}
	Let $\hat{\Up}(x):= \tup(x)u_{\bsig_0}$, where $\tup$ and $\bsig_0$ are as in Theorem \ref{thm; rigidity; AM part main thm}. Then there exists $q_0 \in G$ such that $\Gamma_2 \cap q_0\inv \Gamma_1 q_0$ has finite index in $\Gamma_2$ and satisfying: if $\gamma_i \in \Gamma_2$, $1 \le i \le \ell$, are such that $\Gamma_1 q_0 \Gamma_2 = \bigcup\limits_{1\le i \le \ell} \Gamma_1 q_0 \gamma_i$, then $$\hat{\Up}(\Gamma_2 g) = \{\Gamma_1 q_0 \gamma_i g : 1 \le i \le \ell\}$$ on a $\br$-conull subset of $X_2$. Moreover, the joining $\mu$ is a $\Delta(U)$-invariant measure supported on $\{(\hat{\Up}(x_2),x_2) : x_2 \in X_2\}$, and hence is a finite cover self-joining as in Definition \ref{defn; finite cover self joining}.
\end{theorem}

\subsection{Proof of Theorem \ref{thm; final classification}}\label{section; final joining classification}

We show in this section how to use Theorem \ref{thm; 717} to prove Theorem \ref{thm; finite cover}, which is a more precise statement of Theorem \ref{thm; final classification}. 

By Theorem \ref{thm; 717}, there exists a $\br$-conull set $\tilde{X}_2$ and a natural number $\ell>0$ such that $$\Up(x) := \pi_1(\pi_2\inv(x_2))$$ has cardinality $\ell$ for all $x \in \tilde{X}_2$. Moreover, this is a $U$-equivariant condition, so we may assume that $\tilde{X}_2$ is $U$-invariant and that $$\Up(xu)=\Up(x)u$$ for all $x \in \tilde{X}_2$ and $u \in U$. By a standard argument for constructing cross sections, there exist measurable maps $\up_1,\ldots, \up_\ell :\tilde{X}_2 \to X_1$ such that $$\Up(x) = \{\up_1(x),\ldots,\up_\ell(x)\}.$$

The proof of Theorem \ref{thm; finite cover} now follows as in \cite[Prop. 7.23]{joinings}, where references to Theorems 6.1 and equation (7.21) are replaced with references to Theorems \ref{thm; rigidity; AM part main thm} and \ref{thm;rigidity;U- main thm}. %This completes the proof of Theorem \ref{thm; final classification}.

\subsection{Notation} \label{section; joining notation}
We provide here for the readers convenience a list of important notation that will be used in \S\ref{section; fibers finite}. Full explanations appear in that section, which should be read first, using this section for reference when needed.

\subsubsection*{The measure $\mu$} $\mu$ is an ergodic $U$-joining on $X = X_1 \times X_2$ for the pair $(\br_1, \br_2)$.% Note that we will often omit the subscripts on $\br$.
\subsubsection*{$\psi$ and $\Psi$} $\psi \in C_c(X_1)$ is non-negative with $\br(\psi)>0$. Let $$\Psi = \psi \circ \pi_1 \in C(X).$$
\subsubsection*{The set $\Omega_1$} A compact set $\Omega_1 \subseteq \{x:x^-\in\Lambda_{\operatorname{r}}(\Gamma)\}$ with $\br(\Omega_1)>0$ such that  \[\lim\limits_{T \to \infty} \frac{1}{\ps_x(B_U(T))} \int_{B_U(T)} \psi(xu_\t)d\t = \br(\psi)\] holds uniformly for all $x \in \Omega_1$. %(See \cite[Remark 4.8]{joinings}.)
\subsubsection*{The set $Q$} A compact set $Q \subseteq X$ with $\mu(Q) \in (0,\infty)$, $\pi_1(Q) \subseteq \Omega_1$, and such that for all $f \in C_c(X)$ and all $x \in Q$, \begin{equation}\lim\limits_{T \to \infty} \frac{\int_{B_U(T)} f(x\Delta(u_\t))d\t}{\int_{B_U(T)} \Psi(x\Delta(u_\t))d\t} = \frac{\mu(f)}{\mu(\Psi)}.\end{equation} %Such a set exists by the Hopf ratio ergodic theorem.
\subsubsection*{$\eps$, $\eta_0$ and the sets $Q_+, Q_{++}$} Fix $0<\eps<1$ satisfying $$(1+2\eps)^{-2} > 1/2$$ and $\eta_0>0$ such that $$\mu(Q_{++}) \le (1+\eps)\mu(Q),$$ where \[Q_{++} := Q (B(\eta_0)\times B(\eta_0)) = Q\{(g,g) \in G \times G : \|g-I\| \le \eta_0\}.\] Define \[Q_+ := Q (B(\eta_0/4)\times B(\eta_0/4)) = Q\{(g,g) \in G\times G : \|g-I\|\le \eta_0/4\}.\]
\subsubsection*{$\phi$ and $\Phi$} $\phi \in C_c(X_1)$ is such that $$\1_{\pi_1(Q_{++})} \le \phi \le 1,$$ where $\1_E$ denotes the characteristic function of a set $E$ in $X$. Let $\Phi = \phi \circ \pi_1$. 
\subsubsection*{The set $Q_\eps$ and the family $\mathcal{F}$} Let $\mathcal{F} = \{\1_Q, \1_{Q_+}, \1_{Q_{++}}, \Phi\}$. $Q_\eps\subseteq Q$ is a compact set with $$\mu(Q_\eps)>(1-\eps)\mu(Q)$$ such that for each $f \in C_c(X) \cup \mathcal{F}$ and $\theta>0$, there exists $T_0 = T_0(f,\theta)$ such that if $T\ge T_0$, then \[\left|\frac{\int_{B(T)} f(x\Delta(u_\t))d\t}{\int_{B(T)}\Psi(x \Delta (u_\t))d\t} - \frac{\mu(f)}{\mu(\Psi)}\right| \le \theta\] for all $x \in Q_\eps$. 
\subsubsection*{The functions $\varphi_m$} Suppose that we have a sequence $g_m \in G-U$ with $g_m \to 1_G$ and a point $x=(x_1,x_2) \in Q_\eps$ such that $(x_1g_m,x_2)\in Q_\eps$ for all $m$. For each $m \ge 0$, $$\varphi_m(\t) := u_\t\inv g_m u_\t.$$ %In particular, $\varphi_m(\t)$ satisfies $$x\Delta(u_\t) = x(\varphi_m(\t),1_G)\Delta(u_\t).$$
\subsubsection*{The values $T_m$} Define $T_m := \sup\{T > 0 : \varphi_m(B_U(T))\subseteq B(1)\}$. %Since $g_m \not\in U = N_G(U)$ (where (U)$ denotes the normalizer in $G$ of $U$), $\varphi_m(\t)$ is not constant, and so $T_m < \infty$. Moreover, $T_m \to \infty$ because $g_m \to 1_G$.
\subsubsection*{The functions $\tilde{\varphi}_m$ and $\tilde{\varphi}$} On $B_U(1)$, define $$\tilde{\varphi}_m(\t) = \varphi_m(T_m \t).$$ By equicontinuity of the entries of the $\tilde{\varphi_m}$'s, we may assume that there exists some $\tilde{\varphi}$ defined on $B_U(1)$ such that $\tilde{\varphi}_m \to \tilde{\varphi}$ uniformly on $B_U(1)$. %Observe that $\tilde{\varphi}$ maps into $N_G(U) = U$ by construction of the $\varphi_m$'s.
\subsubsection*{The sets $N_{\eps'}(\tilde{\varphi}-1_G)$, and $I_{\eps'}(T)$} For $\eps'>0$, define $$N_{\eps'}(\tilde{\varphi}-1_G) := \{\t \in B_U(1) : \|f(\t)-1_G\|<\eps'\}.$$

For $T>0$, let $s = \log T$ and define $$I_{\eps'}(T) = B_U(T) - a_{-s}N_{\eps'}(\tilde{\varphi}-1_G)a_s.$$

\subsection{Fibers of $\pi_2$ are finite} \label{section; fibers finite}

Recall that $\mu$ is an ergodic $U$-joining on $X = X_1 \times X_2$ for $(\br_1,\br_2)$. Let \begin{itemize}
	\item $\psi \in C_c(X_1)$ be non-negative with $\br(\psi)>0$,
	\item and let $\Psi = \psi \circ \pi_1 \in C(X)$.
\end{itemize}

Recall that $$N_{1/2}(\supp\bms):=(\supp\bms)B_U(1/2).$$

Let $\Omega_1 \subseteq \{x:x^-\in\Lambda_{\operatorname{r}}(\Gamma)\}$ be a compact set with $\br(\Omega_1)>0$ such that \cite[Lemma 4.6]{joinings} holds for $\psi$ uniformly across all $x \in \Omega_1$. That is, the convergence \begin{equation}\lim\limits_{T \to \infty} \frac{1}{\ps_x(B_U(T))} \int_{B_U(T)} \psi(xu_\t)d\t = \br(\psi)\label{eqn;defn of omega1}\end{equation} holds uniformly for all $x \in \Omega_1$. (By Egorov's theorem, such a compact set exists within any set with positive $\br$ measure, see \cite[Remark 4.8]{joinings}.)

By the Hopf ratio ergodic theorem, there exists a compact set $Q \subseteq X$ such that \begin{itemize}
	\item $\mu(Q) \in (0,\infty)$,
	\item $\pi_1(Q) \subseteq \Omega_1$,
	\item and for all $f \in C_c(X)$ and all $x \in Q$, \begin{equation}\lim\limits_{T \to \infty} \frac{\int_{B_U(T)} f(x\Delta(u_\t))d\t}{\int_{B_U(T)} \Psi(x\Delta(u_\t))d\t} = \frac{\mu(f)}{\mu(\Psi)}.\label{eqn beginning hopf}\end{equation} 
\end{itemize} 

Fix $0<\eps<1$ satisfying $$(1+2\eps)^{-2} > 1/2$$ (this condition is needed to ensure a non-empty intersection in the claim in the proof of Theorem \ref{thm; 712}) and $\eta_0>0$ such that $$\mu(Q_{++}) \le (1+\eps)\mu(Q),$$ where \[Q_{++} := Q (B(\eta_0)\times B(\eta_0)) = Q\{(g,g) \in G \times G : \|g-I\| \le \eta_0\}.\] Also define \[Q_+ := Q (B(\eta_0/4)\times B(\eta_0/4)) = Q\{(g,g) \in G\times G : \|g-I\|\le \eta_0/4\}.\]

Let $\phi \in C_c(X_1)$ be such that $$\1_{\pi_1(Q_{++})} \le \phi \le 1,$$ where $\1_E$ denotes the characteristic function of a set $E$ in $X$. Let $$\Phi = \phi \circ \pi_1,$$ and define 
$$\mathcal{F} = \{\1_Q, \1_{Q_+}, \1_{Q_{++}}, \Phi\}.$$ By the Hopf ratio ergodic theorem again together with Egorov's theorem, there exists a compact set $$Q_\eps\subseteq Q$$ with $$\mu(Q_\eps)>(1-\eps)\mu(Q)$$ such that for each $f \in C_c(X) \cup \mathcal{F}$, the convergence in equation (\ref{eqn beginning hopf}) holds uniformly for all $x \in Q_\eps$. That is, for all $f \in C_c(X)\cup\mathcal{F}$ and $\theta>0$, there exists $T_0 = T_0(f,\theta)$ such that if $T\ge T_0$, then \[\left|\frac{\int_{B(T)} f(x\Delta(u_\t))d\t}{\int_{B(T)}\Psi(x \Delta (u_\t))d\t} - \frac{\mu(f)}{\mu(\Psi)}\right| \le \theta\] for all $x \in Q_\eps$.

The proof of Theorem \ref{thm; 717} will follow as in the proof of \cite[Theorem 7.17]{joinings} once we establish the following generalization of \cite[Theorem 7.12]{joinings}:

\begin{theorem}[c.f. \cite{joinings}, Theorem 7.12]
	Suppose that there exists $x=(x_1,x_2) \in Q_\eps$ and a sequence $g_m \in G-U$ with $g_m \to 1_G$ such that $(x_1g_m,x_2)\in Q_\eps$ for all $m$. Then $\mu$ is invariant under a nontrivial connected subgroup of $U \times \{1_G\}$.
	\label{thm; 712}
\end{theorem}

The proof of Theorem \ref{thm; 712} requires several lemmas, which in turn require more setup.

Suppose that we have a sequence $g_m \in G-U$ with $g_m \to 1_G$ and a point $x=(x_1,x_2) \in Q_\eps$ such that $(x_1g_m,x_2)\in Q_\eps$ for all $m$. For each $m \ge 0$, define $$\varphi_m(\t) := u_\t\inv g_m u_\t.$$ In particular, $\varphi_m(\t)$ satisfies $$x\Delta(u_\t) = x(\varphi_m(\t),1_G)\Delta(u_\t).$$

Define $$T_m := \sup\{T > 0 : \varphi_m(B_U(T))\subseteq B(1)\}.$$ Since $g_m \not\in U = C_G(U)$ (where $C_G(U)$ denotes the centralizer in $G$ of $U$), $\varphi_m(\t)$ is not constant, and so $T_m < \infty$. Moreover, $$T_m \to \infty$$ because $g_m \to 1_G$.

On $B_U(1)$, define $$\tilde{\varphi}_m(\t) = \varphi_m(T_m \t).$$ By definition of $T_m$, each of the entries in the $\tilde{\varphi}_m$'s gives rise to a sequence of uniformly bounded polynomials with degree at most 2 on a compact domain, hence an equicontinuous family. Thus, we may assume that there exists some $\tilde{\varphi}$ defined on $B_U(1)$ such that $$\tilde{\varphi}_m \to \tilde{\varphi}$$ uniformly on $B_U(1)$. Observe that $\tilde{\varphi}$ maps into $C_G(U) = U$ by construction of the $\varphi_m$'s, so $\tilde{\varphi}\in\F_{2,k}$ for some $k>0$, where $\F_{2,k}$ is defined in \S\ref{section; varieties}.

Define $$N_{\eps'}(\tilde{\varphi}-1_G)= \{\t \in B_U(1):\|\tilde{\varphi}(\t)-1_G\|<\eps'\}.$$ For $T>0$, let $s = \log T$ and define $$I_{\eps'}(T) = B_U(T) - a_{-s}N_{\eps'}(\tilde{\varphi}-1_G)a_s.$$

%Define $$\tilde{V} = \{\t \in B_U(1): \|\tilde{\varphi}(\t) - I \| = 0\},$$ and for $\eps'>0$, $$N_{\eps'}(\tilde{V}) := \tilde{V}B_U(\eps')\cap B_U(1)= \{v u_\t : v \in \tilde{V}, \t \in B_U(\eps')\}\cap B_U(1).$$ 

\begin{lemma}
	For every $0<\eta'<1/2$, there exists $\eps'>0$ and $T_0>0$ such that for all $T \ge T_0$, for all $F \in \{\Psi, \Phi\}$, and for all $y \in Q_\eps$, we have that \[\frac{\int_{a_{-s}N_{\eps'}(\tilde{\varphi}-1_G)a_{s}}F(y\Delta(u_\t))d\t}{\int_{I_{\eps'}(T_m)} \Psi(y\Delta(u_\t))d\t} \le c_1\eta',\] for some constant $c_1 \ge 1$, where $s = \log T$.% and $I_{\eps'}(T) := B_U(T) - a_{-s}N_{\eps'}(\tilde{V})a_{s}$.
	\label{lem; quotient over psi on N unif small}
\end{lemma}
\begin{proof}
	Let $\kappa = \kappa(\Omega_1)$ be as in Lemma \ref{lem;PS inf kappa for compact set}. Since $\pi_1(Q_\eps) \subseteq \Omega_1$, it follows from equation (\ref{eqn;defn of omega1}) that there exists $T_1 \ge \kappa$ such that for all $T \ge T_1$ and all $y \in Q_\eps$, \begin{align}\label{eqn lower bound for psi}\int_{B_U(T)} \Psi(y\Delta(u_\t))d\t = \int_{B_U(T)}\psi(\pi_1(y)u_\t)d\t &\ge \frac{1}{2}\ps_{\pi_1(y)}(B_U(T))\br(\psi)>0.\end{align}
	
%	By equation (\ref{eqn;defn of omega1}) and since $\pi_1(Q_\eps)\subseteq \Omega_1$, there exists $T_1\ge \kappa(\Omega_1)$, where $\kappa(\Omega_1)$ is as in Lemma \ref{lem;PS inf kappa for compact set}, such that for all $T\ge T_1$ and for all $y \in Q_\eps$, $T \ge T_1$ implies 
	
	By definition of $\kappa$, for all $x \in \Omega_1,$ $$xB_U(\kappa)\cap \supp\bms \ne \emptyset.$$ From this, it follows that there exists $T_2 \ge T_1$ such that for all $T \ge T_2$, $$xa_{-\log T} \in N_{1/2}(\supp\bms) \cap \{x : x^- \in \Lambda_{\operatorname{r}}(\Gamma)\},$$ where $N_{1/2}(\supp\bms):=(\supp\bms)B_U(1/2).$

	Now, let $f = \psi$ if $F = \Psi$, and $f = \phi$ if $f=\Phi$. Fix $0<\eta'<1$ and let $$\eta = \frac{1}{2D}\eta'\br(f),$$ where $D$ is the implied constant from Lemma \ref{lem;int over V small, abs cty} applied to $f$ and $N_{\eps'}(\tilde{\varphi}-1_G)$. That is, there exists $\eps'>0$ and $T_0 \ge T_2$ such that for all $T \ge T_0$ and all $w \in \Omega_1$, $$e^{-\delta_{\Gamma_1} s}\int_{a_{-s}N_{\eps'}(\tilde{\varphi}-1_G)a_s} f(wu_\t)d\t \le D\eta \ps_{wa_{-s}}(B_U(1)),$$ where $s=\log T$. In particular, this implies that for all $T \ge T_0$ and $w \in  \Omega_1$, \begin{equation}\int_{a_{-s}N_{\eps'}(\tilde{\varphi}-1_G)a_s} f(wu_\t)d\t \le \frac{1}{2}\eta' \br(f)\ps_{w}(B_U(T)),\label{eqn upper bound for phi,psi}\end{equation} where we have used that $\ps_w(B_U(T)) = e^{\delta_{\Gamma_1} s}\ps_{wa_{-s}}(B_U(1))$.
	
	By subtracting equation (\ref{eqn upper bound for phi,psi}) for $\psi$ from (\ref{eqn lower bound for psi}), we conclude that for all $T \ge T_0$ and for all $w \in \Omega_1$, \begin{equation}
	\int_{I_{\eps'}(T_m)} \psi(wu_\t)d\t \ge \frac{1}{2}(1-\eta') \ps_w(B_U(T))\br(\psi). \label{eqn final lower bound psi}
	\end{equation}
	
	Then from equations (\ref{eqn upper bound for phi,psi}) and (\ref{eqn final lower bound psi}), we have that for all $T \ge T_0$ and $y \in \Omega_1$, \[\frac{\int_{a_{-s}N_{\eps'}(\tilde{\varphi}-1_G)a_s} f(yu_\t)d\t}{\int_{I_{\eps'}(T_m)} \psi(yu_\t)d\t} \le c_1 \frac{\eta'}{1-\eta'} \le c_1 \eta',\] for $c_1 = \max\{\br(\phi)/\br(\psi), 1\}$, as desired.	
\end{proof}

By definition of $Q_\eps$, we have that for all $F \in \mathcal{F}$ and for all $\theta>0$, there exists $T_0 = T_0>0$ such that if $T \ge T_0$, then for all $y \in Q_\eps$, \begin{equation} \left|\frac{\int_{B_U(T)} F(y\Delta(u_\t))d\t}{\int_{B_U(T)}\Psi(y \Delta (u_\t))d\t} - \frac{\mu(F)}{\mu(\Psi)}\right| \le \theta. \label{eqn uniform Qeps}\end{equation} We can now improve this to integration over sets of the form $I_{\eps'}(T)$ as follows.

\begin{corollary}
	For all $\theta>0$, there exists $\eps'>0$ and $T_0>0$ such for all $T \ge T_0$, all $y \in Q_\eps$, and every $F \in \mathcal{F} = \{\1_Q, \1_{Q_+}, \1_{Q_{++}}, \Phi\}$, we have that \[\left|\frac{\int_{I_{\eps'}(T)} F(y\Delta(u_\t))d\t}{\int_{I_{\eps'}(T)}\Psi(y \Delta (u_\t))d\t} - \frac{\mu(F)}{\mu(\Psi)}\right| \le \theta.\]
	\label{cor uniformity Qeps}
\end{corollary}
\begin{proof}
	Let $\eps'$ and $T_0$ be as in Lemma \ref{lem; quotient over psi on N unif small}. Let $T \ge T_0$ and let $s = \log T$. Define $N_s = a_{-s}N_{\eps'}(\tilde{\varphi}-1_G)a_s \cap B_U(T)$. By equation (\ref{eqn uniform Qeps}), there exists $\Theta(T)$ which tends to zero uniformly over $Q_\eps$ such that \begin{align*}
	&\int_{B_U(T)}F(y\Delta(u_\t))d\t \\
	&=\left(\frac{\mu(F)}{\mu(\Psi)}+\Theta(T)\right)\int_{B_U(T)}\Psi(y\Delta(u_\t))d\t\\
	&= \frac{\mu(F)}{\mu(\Psi)}\left(\int_{I_{\eps'}(T)} \Psi(y\Delta(u_\t))d\t +  \int_{N_s} \Psi(y\Delta(u_\t))d\t\right) + \Theta(T)\int_{B_U(T)}\Psi(yu_\t)d\t.
	\end{align*} 
	
	Thus, by subtracting $\int_{N_s}F(y\Delta(u_\t))d\t$ and dividing by $\int_{I_{\eps'}(T)}\Psi(y\Delta(u_\t))d\t$, we obtain \begin{align*}
	\frac{\int_{I_{\eps'}(T)} F(y\Delta(u_\t))d\t}{\int_{I_{\eps'}(T)} \Psi(y\Delta(u_\t))d\t} 
	&= \frac{\mu(F)}{\mu(\Psi)}\left(1+\frac{\int_{N_s} \Psi(y\Delta(u_\t))d\t}{\int_{I_{\eps'}(T)} \Psi(y\Delta(u_\t))d\t}\right) \\&\text{ }+ \Theta(T)\left(1+\frac{\int_{N_s} \Psi(y\Delta(u_\t))d\t}{\int_{I_{\eps'}(T)} \Psi(y\Delta(u_\t))d\t}\right) - \frac{\int_{N_s} F(y\Delta(u_\t))d\t}{\int_{I_{\eps'}(T)} \Psi(y\Delta(u_\t))d\t}
	\end{align*}
	
	The conclusion then follows from Lemma \ref{lem; quotient over psi on N unif small}, where for the last term we note that for all $f \in \mathcal{F}, 0\le F \le \Phi$.
\end{proof}

We can now prove Theorem \ref{thm; 712}.

\begin{proof}[Proof of Theorem \ref{thm; 712}]
	Recall the notation from \S\ref{section; joining notation}. Define $$T_m' = \sup\{\tau > 0 : \varphi_m({B_U(\tau)}) \subseteq {B(\eta_0/4)}\}.$$ Note that $T_m\to \infty$ as $m \to \infty$.
	
	It follows from Corollary \ref{cor uniformity Qeps} (by writing out with error terms and dividing) that there exists $\eps'>0$ and $T_0>0$ such that for every $T \ge T_0$, every $y \in Q_\eps$, and every $F_1,F_2 \in \mathcal{F} = \{\1_Q, \1_{Q_+},\1_{Q_{++}}\}$, \begin{equation}
	\left|\frac{\int_{I_{\eps'}(T)} F_1(y\Delta(u_\t))d\t}{\int_{I_{\eps'}(T)} F_2(y\Delta(u_\t))d\t} - \frac{\mu(F_1)}{\mu(F_2)}\right| \le \eps.
	\label{eqn uniformity in 7 12}
	\end{equation} Moreover, this $T_0$ can be chosen so that $\lambda(\{\t \in I_{\eps'}(T_0): x\Delta(u_\t)\in Q_{++}\}) > 0$, where $\lambda$ denotes the Lebesgue measure on $U$.
	\newline
	
	\textbf{Claim:} Let $h_m = (g_m,1_G)$. For all $m$ with $T_m' \ge T_0$ and all $T_0 \le T \le T_m'$, \[\{\t \in I_{\eps'}(T) : x\Delta(u_\t), xh_m\Delta(u_\t)\in Q\} \ne \emptyset.\]
	\begin{proof}[Proof of claim]
		
		Recall that $\varphi_m(\t) = u_\t\inv g_m u_\t$ satisfies $$xh_m\Delta(u_\t) = x\Delta(u_\t)(\varphi_m(\t),1_G).$$ By definition of $T_m'$, if $|\t| \le T_m'$, then $\|(\varphi_m(\t),1_G) - (1_G,1_G)\|\le \eta_0/4$, so \begin{align*}
		\{\t \in I_{\eps'}(T) : x\Delta(u_\t) \in Q\}&\subseteq \{\t \in I_{\eps'}(T) : xh_m\Delta(u_\t) \in Q_+\} \\
		&\subseteq \{\t \in I_{\eps'}(T) : x\Delta(u_\t)\in Q_{++}\}
		\end{align*}
		
		By applying equation (\ref{eqn uniformity in 7 12}) to $F_1 = \1_{Q_{++}}$ and $F_2 = \1_Q$ with $y = x$, we have that \[\lambda(\{\t \in I_{\eps'}(T) : x\Delta(u_\t) \in Q\}) \ge (1+2\eps)\inv \lambda(\{\t \in I_{\eps'}(T) : x\Delta(u_\t)\in Q_{++}\}),\] where $\lambda$ is the Lebesgue measure. 
		
		And by applying it with $F_1 = \1_{Q_{+}}$ and $F_2 = \1_{Q}$ with $y = xh_m$, we have that \begin{align*}
		\lambda(\{\t \in I_{\eps'}(T) : xh_m\Delta(u_\t)\in Q\}) &\ge (1+2\eps)\inv \lambda(\{\t \in I_{\eps'}(T): xh_m\Delta(u_\t) \in Q_{+}\})\\
		&\ge (1+2\eps)\inv \lambda(\{\t \in I_{\eps'}(T): x \Delta(u_\t)\in Q\}) \\
		&\ge (1+2\eps)^{-2} \lambda(\{\t \in I_{\eps'}(T) : x\Delta(u_\t) \in Q_{++}\})
		\end{align*}
		
		Since $\{\t \in I_{\eps'}(T) : x\Delta(u_\t)\in Q\}$ and $\{\t \in I_{\eps'}(T) : xh_m\Delta(u_\t) \in Q\}$ are both subsets of $\{\t \in I_{\eps'}(T) : x\Delta(u_\t) \in Q_{++}\}$ from the definition of $T_m'$, the choice of $\eps$ implies that both subsets have greater than half the Lebesgue measure of the larger set (which is positive by choice of $T_0$), and thus their intersection cannot be empty.
	\end{proof}

	By the claim, for all sufficiently large $m$, there exists $\t_m \in I_{\eps'}(T_m')$ such that $$x\Delta(u_{\t_m}) \in Q  \text{ and }xh_m\Delta(u_{\t_m}) = x\Delta(u_{\t_m}) (\varphi_m(\t_m),1_G) \in Q.$$ By the compactness of $Q$ and by dropping to a subsequence if necessary, we may assume that there exists $x_\infty \in Q$ such that $x\Delta(u_{\t_m}) \to x_\infty$. 
	
	Let $$\tilde{\t}_m = \t_m / T_m' \in {B_U(1)},$$ so that $$\varphi_m(\t_m) = \tilde{\varphi}_m(\tilde{\t}_m).$$ Then again by the compactness of ${B_U(1)}$ and the uniform convergence of $\tilde{\varphi}_m \to \tilde\varphi$, we may assume that there exists $\t_\infty \in {B_U(1)}$ such that $$\varphi_m(\t_m) \to \tilde{\varphi}(\t_\infty).$$ 
	
	By definition of $I_{\eps'}(T_m')$, $\tilde{\t}_m \not\in N_{\eps'}(\tilde{\varphi}-1_G)$ for all $m$, so $$\tilde\varphi(\t_\infty) \ne 1_G.$$ Moreover, the image of $\tilde{\varphi}$ is contained within $C_G(U)=U$, so it follows from \cite[Lemma 7.7]{joinings} applied to $(\tilde\varphi(\t_\infty),1_G)$ that $\mu$ is quasi-invariant under a nontrivial connected subgroup of $U \times \{1_G\}$. Strict invariance follows from \cite[Lemma 7.3]{joinings}.	
\end{proof}

%The proof of Theorem \ref{thm; 717} now follows as in \cite[Theorem 7.17]{joinings}, with Theorem \ref{thm; 712} replacing references to Theorem 7.12 in that proof.

%For completeness, we include the following proof of Theorem \ref{thm; 717}. This is essentially the same as the proof of \cite[Theorem 7.17]{joinings}, with Theorem \ref{thm; 712} replacing references to Theorem 7.12 in that proof.

We now prove Theorem \ref{thm; 717}, following the approach of \cite[Theorem 7.17]{joinings}.

\begin{proof}[Proof of Theorem \ref{thm; 717}.]
	We begin by showing that $\br$-a.e.\ fiber measure $\mu_{x_2}^{\pi_2}$ is atomic. Define $$B = \{x_2 \in X_2 : \mu_{x_2}^{\pi_2} \text{ is not purely atomic}\},$$ and assume for contradiction that $\br(B)>0$. Then, as in \cite[Remark 4.8]{joinings}, we may find a compact set $$\Omega_1 \subseteq B \cap \{x:x^- \in \Lambda_r(\Gamma)\}$$ with $\br(\Omega_1)>0$ and satisfying (\ref{eqn;defn of omega1}). 
	
	Write $$\mu_{x_2}^{\pi_2} = (\mu_{x_2}^{\pi_2})^a + (\mu_{x_2}^{\pi_2})^c,$$ where $(\mu_{x_2}^{\pi_2})^a$ is the purely atomic part and $(\mu_{x_2}^{\pi_2})^c$ is the continuous part. Define $$\mathcal{B} = \{(x_1,x_2):x_2 \in B, x_1 \in \supp((\mu_{x_2}^{\pi_2})^c)\}.$$ Then there exists $$Q \subseteq \mathcal{B}$$ compact with $\mu(Q) \in (0,\infty)$, $\pi_1(Q)\subseteq \Omega_1$, and satisfying equation (\ref{eqn beginning hopf}) as in the beginning of this section, and similarly can define $Q_\eps \subseteq Q$.
	
	Since $Q_\eps \subseteq \mathcal{B}$, there exists $x = (x_1,x_2)\in Q_\eps$ and a sequence $(x_{1,m},x_2) \in Q_\eps$ with $x_{1,m} \ne x_1$ and $$(x_{1,m},x_2) \to x.$$ We will show that this implies that $\mu$ is invariant under a non-trivial connected subgroup of $U \times\{1_G\}$, a contradiction to \cite[Lemma 7.16]{joinings}.
	
	Write $$(x_{1,m},x_2) = x(g_m,1_G)$$ where $g_m \to 1_G$, $g_m \ne 1_G$. There are two possible cases.
	
	First, suppose that $g_m \in U$ for all sufficiently large $m$. Then by \cite[Lemma 7.7]{joinings}, $\mu$ will be quasi-invariant under the subgroup generated by $\{(g_m,1_G)\}$, which implies invariance under a non-trivial connected subgroup of $U\times\{1_G\}$ becauase $g_m \to 1_G$ and $U$ is unipotent. This is a contradiction.
	
	Thus, it must be that there exists a subsequence $g_{m_k} \not\in U$ for all $m_k$. Then by Theorem \ref{thm; 712}, $\mu$ is invariant under a nontrivial connected subgroup of $U\times \{1_G\}$, again a contradiction.
	
	In all cases, we obtain a contradiction, and so it must have been that $$\br(B)=0,$$ that is, $\br$-a.e.\ fiber measure is atomic. Now, define $$Z = \left\{(x_1,x_2)\in X: \mu_{x_2}^{\pi_2}(\{x_1\}) = \max\limits_{y \in \pi_2\inv(x_2)}\mu_{x_2}^{\pi_2}(\{y\})\right\}.$$ We have shown that $\br$-a.e.\ fiber measure is atomic, and $Z$ is $\Delta(U)$ invariant, so it follows from ergodicity of the joining $\mu$ that $\mu(Z^c)=0$. This implies that there exists some $\ell \in \N$ so that $\br$-a.e.\ $x_2$ has $$|\pi_1(\pi_2\inv(x_2))|=\ell,$$ and the fiber measure $\mu_{x_2}^{\pi_2}$ is the uniform distribution on $\ell$ points, as desired.
\end{proof}

\end{document}